\documentclass[10pt, twoside]{article}
\usepackage[english]{babel}
\usepackage[utf8]{inputenc}
\usepackage[a4paper,top=2.25cm,bottom=2.5cm,left=2.25cm,right=2.25cm]{geometry}
\usepackage{indentfirst}
\usepackage{hyperref}
\hypersetup{
    colorlinks,
    allcolors={black},
}

\usepackage{amsmath, amsfonts, amssymb, amsthm, mathtools}
\usepackage{dsfont}
\usepackage{physics}

\usepackage{graphicx}
\usepackage{subcaption}
\usepackage{float}
\usepackage{tikz, pgfplots, quiver}

\usepackage{algorithm}
\usepackage{algorithmic}

\usepackage{enumitem}
\usepackage{booktabs}
\usepackage{mdframed}
\usepackage{siunitx}
\usepackage{url}
\usepackage{cleveref}

\usepackage{authblk}

\newtheorem{theorem}{Theorem}[section]

\newtheorem{proposition}[theorem]{Proposition}
\newtheorem{lemma}[theorem]{Lemma}

\theoremstyle{definition}
\newtheorem{definition}[theorem]{Definition}

\theoremstyle{remark}
\newtheorem{remark}[theorem]{Remark}

\newcommand{\R}{\mathbb{R}}
\newcommand{\N}{\mathbb{N}}

\newcommand{\T}{\mathcal{T}}
\newcommand{\C}{\mathcal{C}}
\newcommand{\F}{\mathcal{F}}
\newcommand{\U}{\mathcal{U}}

\newcommand{\f}{\mathbf{f}}

\newcommand{\control}{\mathbf{u}}
\newcommand{\bfphi}{\boldsymbol{\varphi}}

\newmdenv[
  backgroundcolor=white,
  linecolor=black,
  linewidth=0.5pt,
  skipabove=12pt,
  skipbelow=12pt,
  innertopmargin=6pt,
  innerbottommargin=6pt,
  innerleftmargin=10pt,
  innerrightmargin=10pt,
  frametitle={\begin{center}\bfseries Assumptions H\end{center}},
]{assumptionbox*}

\newcounter{assbox}
\crefformat{assbox}{#2Assumptions~H#3}
\Crefformat{assbox}{#2Assumptions~H#3}

\newenvironment{assumptionbox}{%
  \refstepcounter{assbox}%
  \begin{assumptionbox*}%
}{%
  \end{assumptionbox*}%
}

\newlist{hypotheses}{enumerate}{1}
\setlist[hypotheses,1]{label=\textbf{(H\arabic*)}, ref=\textbf{(H\arabic*)}}

\title{Optimal control of the coagulation-fragmentation equation}
\date{}

\author[1,2]{Enrico Sartor}
\affil[1]{Laboratoire des Signaux et Systèmes, Université Paris-Saclay, CentraleSupélec, CNRS\protect\\
\texttt{enrico.sartor@centralesupelec.fr}}
\affil[2]{Chair for Dynamics, Control, Machine Learning and Numerics, Alexander von Humboldt-Professorship, Department of Mathematics, Friedrich-Alexander-Universit\"at Erlangen-N\"urnberg}

\begin{document}

\maketitle

\begin{abstract}
We formulate and analyse an optimal control problem for the coagulation–fragmentation equation, where a scalar, time-dependent control modulates the coagulation rate by multiplying the coagulation kernel. The objective functional consists of a quadratic penalisation of the control and a terminal cost depending on the final size distribution. In a weighted $L^1$ framework, we prove weak-to-weak continuity of the control-to-state map under perturbations of the coefficients and obtain existence of optimal controls by the direct method. We then establish $\Gamma$-convergence of the corresponding cost functionals, providing stability of optimal controls and justifying truncation of unbounded kernels in the optimisation setting. For bounded coagulation kernels we show differentiability of the dynamics, derive an adjoint equation, and obtain a Pontryagin-type minimum principle. Lipschitz continuity of the gradient with respect to the control yields, at the continuous level, convergence of a projected-gradient algorithm with Armijo backtracking. A proof-of-concept finite-volume implementation is then used in a numerical study targeting the number of particles within a prescribed size window, demonstrating that a single low-dimensional actuator can effectively reshape an infinite-dimensional particle-size distribution.
\end{abstract}

\section{Introduction}\label{sec:intro}

The coagulation-fragmentation equation is a partial differential equation that describes the evolution of particle systems undergoing aggregation and break-up.  It appears in aerosol physics \cite{friedlander2000smoke,seinfeld2016atmospheric}, polymer science\;\cite{van1984size}, blood coagulation\;\cite{samsel1982kinetics,samsel1984kinetics}, and many other disciplines.


Mathematically, letting $f(t,x)$ denote the density of clusters of size $x>0$ at time $t\ge0$, the system's dynamics are then described by the following integro-differential equation
\begin{equation}\label{eqn:  standard cf_equation}
  \partial_t f(t,x)
  = \mathcal{C}f(t,x) + \mathcal{F}f(t,x),
\end{equation}
with the coagulation and fragmentation operators defined respectively as
\begin{equation}\label{eqn: coagulation operator}
      \mathcal{C}f(t,x) = \frac12\int_0^x K(y,x-y)f(t,y)f(t,x-y)\,dy
                      - f(t,x)\int_0^\infty K(x,y)f(t,y)\,dy,
\end{equation}
and
\begin{equation}\label{eqn: fragmentation operator}
      \mathcal{F}f(t,x) = -\alpha(x)f(t,x)
                      + \int_x^\infty \alpha(y)b(x,y)f(t,y)\,dy.
\end{equation}

Here $K(x,y)$ denotes the coagulation kernel, prescribing the rate at which particles of sizes $x$ and $y$ merge to form a particle of size $x+y$. Accordingly, the coagulation operator consists of a gain term, which accounts for the formation of particles of size $x$ as the outcome of binary coalescence of smaller clusters of sizes $y$ and $x-y$, and a loss term, corresponding to their disappearance through coalescence with other particles. Similarly, the fragmentation dynamics are characterised by a break-up rate $\alpha$ and a \emph{daughter distribution} $b$, where $b(x,y)$ denotes the density (with respect to $x$) of producing a fragment of mass $x$ from a particle of mass $y>x$. Thus the fragmentation operator accounts for both the loss of particles due to break-up and the gain of particles of size $x$ arising from the disintegration of larger clusters.

The origins of coagulation modelling trace back to Smoluchowski’s discrete model\;\cite{smoluchowski1916drei,smoluchowski1918versuch}.  Müller soon formulated a continuous-size version~\cite{muller1928allgemeinen}, and Melzak incorporated fragmentation~\cite{melzak1957scalarI,melzak1957scalarII}. The framework most widely used today is due to Vigil and Ziff~\cite{vigil1989stability}. Existence and uniqueness of solutions have been established under varied assumptions; see the survey~\cite{banasiak2019analytic1, banasiak2019analytic}. Current research addresses well-posedness, long-time asymptotics and \emph{self‑similarity}~\cite{carr1992asymptotic,escobedo2005self} as well as \emph{gelation} and \emph{shattering}, the loss of mass through formation of, respectively, an infinite cluster or dust~\cite{hendriks1983coagulation,mcgrady1987shattering}.

\subsection{Motivations}\label{subsec:motivations}

In many applications, coagulation-fragmentation systems are not only used to describe the evolution of particle sizes, but also to steer it toward desired outcomes. Whether the goal is to reduce the concentration of fine particulates, control clot formation, or tune polymer properties, the ability to influence the size distribution has become increasingly important. Two representative examples are the following.

\textbf{Fine‑particle mitigation in aerosols.}  
Sub‑\SI{2.5}{\micro\metre} particulates (PM$_{2.5}$) remain airborne,
penetrate deep into the lungs, and evade ordinary filters.  Raising
relative humidity, injecting trace electrolytes, or altering temperature
promotes droplet coalescence and agglomerates fines into super‑micron
clusters that are more easily captured
\cite{hinds2022aerosol,jacobson2005enhanced}.  Each intervention acts,
to leading order, as a scalar scaling of the collision frequency.

\textbf{Size control in polymer reactors.}  
Surfactants, inhibitors, and initiators tune the rate at which growing
polymer particles merge, thereby fixing latex diameter and molecular
weight distribution
\cite{flory1953principles,matyjaszewski2002handbook}.
Here again the primary effect is a time‑dependent multiplicative factor
on the coagulation kernel.

Both scenarios share key structural features: the control primarily targets coagulation; the system is infinite-dimensional but the control remains low-dimensional; and the goal is typically to optimise a functional of the final size distribution. These characteristics fall naturally within the framework of infinite-dimensional optimal control.

\subsection{The optimal control problem}\label{subsec:formulation}

Motivated by the previous discussion, we introduce a controlled model in which a scalar control 
\begin{equation*}
    \control\in\mathcal{U}\coloneqq \bigl\{\control \in L^2(0,T) \colon \control(t)\in [u_{\min},u_{\max}] \text{ for a.e. } t\in(0,T)\bigr\}.
\end{equation*}
modulates the coagulation rate. Concretely, we replace the original kernel $K(x,y)$ by the time‑dependent kernel
\begin{equation*}
    \hat{K}(t,x,y)\coloneqq\control(t)\,K(x,y).
\end{equation*}
We assume that $0<u_{\min}\leq 1\leq u_{\max}<+\infty$ so that the uncontrolled dynamics correspond to $\control(t)\equiv 1$; choosing $\control(t)>1$ accelerates coagulation, while $\control(t)<1$ slows it down.

With this control in place, the coagulation–fragmentation equation can be rewritten as follows
\begin{equation}\label{eqn: cf_equation}
\partial_t f(t,x)
=
\control(t)\mathcal{C}f(t,x)+\mathcal{F}f(t,x),
\end{equation}
with $\mathcal{C}$ and $\mathcal{F}$ given by~\eqref{eqn: coagulation operator}--\eqref{eqn: fragmentation operator}.  Equation~\eqref{eqn: cf_equation} is a nonlinear nonlocal PDE in which the control acts multiplicatively on the nonlinear (coagulation) part of the dynamics.  For physical reasons, namely a finite number of particles and finite total mass, we study~\eqref{eqn: cf_equation} in a suitably weighted $L^{1}$ space (see Section~\ref{subsec:functional-setting}).

Associated with \eqref{eqn: cf_equation} we consider the optimal control problem
\begin{equation}\label{eqn: OCP}\tag{\bf{OCP}}
  \inf_{\control\in\mathcal{U}}
  \;\frac{w}{2}\int_0^T\bigl(\control(t)-1\bigr)^2\,dt + \psi\bigl(f(T)\bigr),
\end{equation}
where the running cost penalises deviations from the baseline $\control\equiv 1$, that is, the use of the control. The weight $w\geq0$ modulates the strength of this penalisation. A specific terminal cost of practical interest is
\begin{equation*}
    \psi(f)\coloneqq\pm\int_{x_{\min}}^{x_{\max}} f(x)\,dx,
\end{equation*}
so that the objective is to \emph{minimise} or \emph{maximise}  the total number of particles whose sizes lie in the interval $[x_{\min},x_{\max}]$. This simple choice reflects the applications in Section~\ref{subsec:motivations}: suppressing fines in aerosols  or concentrating polymer particles within a desired window. 

\subsection{Contributions and related work}

The main goal of this paper is to analyse the optimal control problem \eqref{eqn: OCP}. At a qualitative level, we prove that, under standard structural assumptions on the coefficients, problem~\eqref{eqn: OCP} admits optimal controls and that these controls are stable under perturbations of the kernels. In the bounded-kernel regime, we further derive a Pontryagin-type minimum principle that characterises optimal controls through an adjoint state evolving in the dual of the natural weighted $L^1$ state space. This leads to an explicit gradient formula, from which we deduce Lipschitz continuity of the reduced gradient and, at the continuous level, convergence of a projected-gradient scheme. Finally, we present an illustrative numerical test case.

\subsubsection{Related work}

Inverse problems for coagulation-fragmentation systems, most notably the
identification of unknown coagulation or fragmentation kernels from
size-distribution data, have been widely studied, both analytically and
numerically, typically via regularisation strategies or short-time
asymptotics; see,  for instance, ~\cite{alomari2013recovery,
doumic2018estimating,doumic2021inverse}.

In a related but distinct direction, optimal and geometric control techniques
have been developed for structured-population and growth--fragmentation models,
in particular in connection with prion amplification and PMCA-type
systems; see, e.g.,~\cite{coron2015optimization,chyba2015optimal}.
These works concern either finite-dimensional approximations or linear
structured dynamics, and therefore do not cover the nonlinear
coagulation-fragmentation equation considered here.

On the applied side, optimal control of population balance models has been
studied extensively in chemical engineering, especially for crystallisation
and granulation processes, where the goal is often to shape particle-size or
crystal-shape distributions; see, for instance, ~\cite{nagy2008distribution,
hofmann2017optimal,de2021optimal}. However, these contributions typically
concern first-order or reduced population-balance models arising in
particle-processing applications, and are primarily oriented toward process
design and numerical optimisation rather than a rigorous optimal-control
analysis of a nonlinear coagulation-fragmentation PDE.

To the best of our knowledge, a rigorous optimal-control theory for nonlinear
coagulation-fragmentation equations treated directly at the PDE level in
weighted $L^1$-type spaces does not seem to be available in the literature.
The present work addresses this gap by formulating and analysing an optimal
control problem directly for the coagulation-fragmentation equation.

\subsubsection{Main contributions}
The analysis proceeds in four steps.

\begin{enumerate}
  \item \textbf{Existence of optimal controls.}
  We first establish weak-to-weak continuous dependence of solutions on the
  control (and on the kernels); see Theorem~\ref{thm: weak dependence on controls}.
  This allows us to prove existence of optimal controls for~\eqref{eqn: OCP}
  by the direct method, provided the terminal cost is weakly lower
  semi-continuous; see Theorem~\ref{thrm: existence of optimal controls}.
  A key technical difficulty is that the control acts multiplicatively on the
  bilinear coagulation operator, so the controlled term is trilinear in
  $(u,f,f)$ and each factor converges only weakly. Overcoming this obstacle
  requires a careful use of compactness in the state space and of a product
  lemma \emph{à la} Egorov (Proposition~\ref{propo: egorov}).

  \item $\boldsymbol{\Gamma}$\textbf{-convergence and stability.}
  The same continuity framework accommodates perturbations of the
  coefficients. We prove $\Gamma$-convergence of the family of cost
  functionals associated with perturbed kernels and show convergence of
  minimisers to those of the limit problem; see Theorem~\ref{thrm: gamma convergence}.
  This extends the classical truncation arguments for unbounded coagulation
  kernels, traditionally used to construct solutions of the PDE, to the optimal
  control setting, and justifies approximating unbounded kernels by bounded
  ones when computing gradients.

  \item \textbf{Adjoint equation, gradient formula, and minimum principle.}
  When the coagulation kernel is bounded, the coagulation operator
  $\mathcal C$ is continuously Fréchet differentiable.
  Linearising the dynamics along a reference trajectory $(\f^\star,\control^\star)$ and
  working in the dual state space, we derive a
  weak-$*$ well-posed backward adjoint equation
  \begin{equation*}
          -\partial_t \varphi^\star(t)
    = \control^\star(t)\,D\mathcal C[\f^\star(t)]^*\varphi^\star(t) + \mathcal F^*\varphi^\star(t),
  \end{equation*}
  whose solution $\bfphi^\star$ yields an explicit gradient representation
  \begin{equation*}
          \nabla_u J(\control^\star)(t)
    = w\bigl(\control^\star(t)-1\bigr)
      + \big\langle \bfphi^\star(t),\,\mathcal C \f^\star(t)\big\rangle,
  \end{equation*}
  see Theorem~\ref{thrm: gradients}. Because the state space is a nonreflexive weighted $L^1$ space, and because the adjoint fragmentation operator is defined only on a weak-$*$ dense domain, the adjoint equation has to be formulated in a mild weak-$*$ sense. From the adjoint system we derive a Pontryagin-type
  minimum principle and a pointwise characterisation of optimal controls as
  projections of the unconstrained minimiser of the Hamiltonian; see
  Theorem~\ref{thm: PMP}.   Under a natural Lipschitz assumption on the terminal cost, we also show that
  the gradient $\nabla_u J\colon\U\to L^2(0,T)$ is Lipschitz continuous (Theorem~\ref{thrm: gradients}). At the continuous level, this regularity yields convergence of the projected-gradient method with Armijo backtracking.

  \item \textbf{Numerical illustration.}
  As a proof-of-concept, we discretise the state and adjoint equations with
  a finite–volume scheme and solve \eqref{eqn: OCP} via projected
  gradient descent with Armijo backtracking. The numerical example,
  inspired by aerosol physics applications, demonstrates that a single
  low-dimensional actuator can substantially reshape the full
  size-distribution within a prescribed window. This section is intended as an illustration only: we do not address discrete adjoint consistency or convergence of the fully discretised optimisation scheme.
\end{enumerate}

Historically, well-posedness for the coagulation-fragmentation equation has mainly been studied through two approaches: the semigroup method, initiated by \cite{aizenman1979convergence}, and the truncation and weak-compactness strategy \emph{à la} Stewart \cite{stewart}. These two approaches are based on different notions of solution: the former relies on mild semigroup solutions, while the latter uses a pointwise formulation; see Definitions~\ref{def: mild solution} and \ref{def: solution}, respectively. In the present work, we adopt the latter notion in order to establish existence of optimal controls and the associated
$\Gamma$-convergence result. We also believe that this formulation is better suited to possible extensions, for instance to size-dependent controls. By contrast, the semigroup framework is the natural setting for the analysis of the linearised dynamics and the adjoint equation. It is therefore necessary, in the bounded-coagulation regime, to prove that these two notions of solution coincide. While such an equivalence is expected, to the best of our knowledge it does not seem to be available in the literature in the weighted setting considered in this paper.

\section{The controlled coagulation–fragmentation equation}

In this section we introduce the functional framework and recall the basic properties
of the controlled coagulation–fragmentation equation
\eqref{eqn: cf_equation}. We fix the state space, specify the notion of
solution, state the standing structural assumptions on the kernels, and
collect the mapping properties of the operators $\C$ and $\F$ together
with a well-posedness result for the controlled dynamics.

\subsection{Functional setting}\label{subsec:functional-setting}

We represent the system by a number density $f\colon (0,\infty)\to \R_{\geq 0}$, where $f(x)$ counts particles of mass $x$.
Physical feasibility requires finite particle number and finite total mass, i.e.,
\begin{equation*}
    M_0(f) \coloneqq \norm{f}_0=\int_0^\infty |f(x)|\,dx < \infty,
    \qquad
    M_1(f) \coloneqq \norm{f}_1=\int_0^\infty x\,|f(x)|\,dx < \infty.
\end{equation*}
It is therefore natural to work in the weighted space
\begin{equation*}
      X_{0,1} \coloneqq L^1\bigl((0,\infty), (1+x) dx\bigr) = X_0 \cap X_1,\quad X_0\coloneqq L^1\bigl((0,\infty)\bigr),\quad X_1\coloneqq L^1\bigl((0,\infty), x dx\bigr)
\end{equation*}
which is a Banach space under
\begin{equation*}
    \norm{f}_{0,1}\coloneqq \int_0^\infty (1+x) \lvert f(x)\rvert dx.
    \end{equation*}
Here $X_0$ controls the particle number $M_0$, $X_1$ controls the mass $M_1$, and their intersection encodes the regime relevant for our controlled dynamics and costs.

We denote by $C([0,T],X_{0,\mathrm w})$ the space of maps $\f:[0,T]\to X_0$ that are continuous with respect to the weak topology of $X_0$, that is, such that for every $\varphi\in X_0^*=L^\infty(0,\infty)$ the scalar map
\begin{equation*}
    t \longmapsto \int_0^\infty f(t,x)\varphi(x)\,dx
\end{equation*}
is continuous on $[0,T]$.

We say that a sequence $(\mathbf \f_n)_{n\in\mathbb N}\subset C([0,T],X_{0,\mathrm w})$ converges uniformly in time in the weak topology of $X_0$ to $\f^\star\in C([0,T],X_{0,\mathrm w})$ if, for every $\varphi\in L^\infty(0,\infty)$,
\begin{equation*}
    \sup_{t\in[0,T]}
    \left|
    \int_0^\infty \bigl(\mathbf \f_n(t,x)-\f^\star(t,x)\bigr)\varphi(x)\,dx
    \right|
    \longrightarrow 0
    \qquad \text{as } n\to\infty.
\end{equation*}

\subsection{The notion of solution and standing assumptions}\label{sec: notion of solution}

We next fix the solution concept \cite{stewart} for the controlled equation and list the structural assumptions on $K$, $\alpha$, and $b$.

\begin{definition}\label{def: solution}
Let $\control\in\mathcal U$ be fixed, and let $f_{\mathrm{in}}\in X_{0,1}$ satisfy $f_{\mathrm{in}}\geq 0$. A \emph{solution} of \eqref{eqn: cf_equation} with initial condition $f_{\mathrm{in}}$ and control $\control$ is a function $\f\colon [0,T]\to X_{0,1}$ such that for a.e.\ $x\in(0,\infty)$:
\begin{itemize}
    \item [\textbf{(S1)}] \label{hyp:s1}$\f(0,x) = f_{\mathrm{in}}(x)$;
    \item [\textbf{(S2)}] \label{hyp:s2}$\f(t,x)\geq 0$ for every $t\in[0,T]$;
    \item [\textbf{(S3)}] \label{hyp:s3} $t\mapsto \f(t,x)$ is continuous;
    \item[\textbf{(S4)}]\label{hyp:s4} for every $t\in[0,T]$
    \begin{equation*}
        \int_0^t \int _0^\infty K(x,y) \f(\tau,y) dy d\tau <+\infty \hspace{0.75cm}\text{ and }\hspace{0.75cm}\int_0^t \int_x^\infty b(x,y) \alpha(y) \f(\tau,y) dy d\tau <+\infty;
    \end{equation*}
    \item[\textbf{(S5)}]\label{hyp:s5} for every $t\in[0,T]$ it holds
    \begin{equation}\label{eqn: solution equation}
        \f(t,x) = f_{\mathrm{in}}(x)+\int _0^t \biggl[\control(\tau) \C \f(\tau,x) + \F \f(\tau,x)\biggr] d\tau.
    \end{equation}
\end{itemize}
\end{definition}

\begin{remark}
Before proceeding, let us comment on Definition~\ref{def: solution}.
Condition \textbf{(S2)} reflects the interpretation of $\f(t,\cdot)$ as a particle
density for every $t\in[0,T]$, and therefore imposes nonnegativity.
Condition \textbf{(S3)} prescribes the time regularity of the solution in the
sense of Stewart \cite{stewart}. Finally, condition \textbf{(S5)} gives the integral formulation
of \eqref{eqn: cf_equation}; in particular, for a.e.~$x\in(0,\infty)$, it
requires that $\tau\mapsto \control(\tau)\,\C\f(\tau,x)+\F\f(\tau,x)$
belongs to $L^1(0,T)$. 

Since the gain contributions are nonnegative, while the loss contributions are
nonpositive, the right-hand side of \eqref{eqn: solution equation} could, a
priori, involve cancellations between terms that are not separately integrable. Condition \textbf{(S4)}, together with \textbf{(S3)}, which yields boundedness of $\f(\cdot,x)$ for a.e.\ $x$, ensures the time integrability of the coagulation-loss term, the fragmentation-loss term, and the fragmentation-gain
term, thereby excluding such ambiguities and making the splitting in
\eqref{eqn: solution equation} meaningful. More precisely, for a.e.~$x\in(0,\infty)$
and every $t\in[0,T]$, \eqref{eqn: solution equation} can be rewritten as
\begin{align*}
    \f(t,x)=f_{\mathrm{in}}(x)
    &+\frac12\int_0^t \control(\tau)\int_0^x
    K(y,x-y)\,\f(\tau,y)\,\f(\tau,x-y)\,dy\,d\tau \\
    &-\int_0^t \control(\tau)\,\f(\tau,x)\int_0^\infty
    K(x,y)\,\f(\tau,y)\,dy\,d\tau \\
    &-\int_0^t \alpha(x)\,\f(\tau,x)\,d\tau
    +\int_0^t \int_x^\infty b(x,y)\,\alpha(y)\,\f(\tau,y)\,dy\,d\tau .
\end{align*}
\end{remark}

We make the following assumptions on the coagulation kernel $K$, fragmentation rate $\alpha$ and daughter distribution $b$.

\begin{assumptionbox}
\label{ass: assumptions}
$K$, $\alpha$ and $b$ satisfy:
\begin{hypotheses}
    \item \label{hyp:symmetry} $K\colon (0,\infty)\times(0,\infty)\to\R$ is measurable and satisfies $K(x,y)=K(y,x)\geq 0$ for every $x,y>0$;
    \item \label{hyp:bound} there exists $K_0\geq 0$ and $0\leq \mu<1$ such that $K(x,y)\leq K_0(1+x)^\mu (1+y)^\mu$ for every $x,y\in(0,\infty)$;
    \item \label{hyp:alpha} $\alpha(x)=\alpha_0 x^\lambda$ with $\alpha_0\geq 0$ and $1> \lambda>\mu$;
    \item \label{hyp:b} there exists $\nu\in (-1,0]$ such that
    \begin{equation*}
        b(x,y) = \frac{2 + \nu}{y}\biggl(\frac{x}{y}\biggr)^{\nu} \mathds{1}_{(0,y)}(x).
    \end{equation*}
\end{hypotheses}
\end{assumptionbox}

Several remarks regarding these assumptions are in order.

\begin{remark}
We assume that the coagulation kernel $K$ is symmetric (see Hypothesis~\ref{hyp:symmetry}), which is natural from a modelling viewpoint: the rate at which a particle of size $x$ coagulates with one of size $y$ should coincide with the rate at which a particle of size $y$ coagulates with one of size $x$. The growth condition in Hypothesis~\ref{hyp:bound}, together with the requirement $\lambda>\mu$ (ensuring that fragmentation dominates coagulation), is customary in order to rule out the well-known phenomenon of gelation \cite{hendriks1983coagulation}, that is, the loss of mass at infinity in finite time. 

The growth assumption on $\alpha$ is standard when working in the space $X_{0,1}$.  For faster fragmentation rates, one typically needs control of
higher moments and therefore a stronger weighted setting, such as
$L^1((0,\infty),(1+x^p)\,dx)$ for some $p>1$, or related interpolation
spaces; see, for instance, \cite[Chapter 5]{banasiak2019analytic1}.
\end{remark}

\begin{remark}
Two elementary consistency checks justify the daughter law in
Assumption~\ref{hyp:b}:
\begin{enumerate}[label=(\roman*)]
  \item \textbf{Local conservation of mass.}  
        During a single break–up the parent mass is only redistributed:
        \begin{equation}\label{eqn:local conservation of mass}
          \int_0^y x\,b(x,y)\,d x = y.
        \end{equation}
  \item \textbf{Finite number of fragments.}  
        The total number of fragments produced by a fragmentation event is
        \begin{equation}\label{eqn:finite number of fragments}
          \int_0^y b(x,y)\,dx=
          \frac{\nu+2}{\nu+1}\coloneqq N_\nu<+\infty.
        \end{equation}
\end{enumerate}
The special case $\nu=0$ yields the classical \emph{binary
fragmentation} kernel $b(x,y)=2/y$, for which $N_0=2$.

\smallskip
Both identities follow from a short calculation (omitted here) and
highlight why the structural conditions on $b$ are essential: they rule
out unphysical scenarios such as mass creation and loss or an infinite burst of
particles from a single fracture.
\end{remark}

\begin{remark}
The fragmentation rate $\alpha$ and daughter distribution $b$ prescribed in
Assumptions~\ref{hyp:alpha} and \ref{hyp:b} form one of the most widely used kernel
families in applications~\cite{ziff1985kinetics, mcgrady1987shattering,  timar2010new, andrejevic2021model}.  Their
appeal rests on three complementary features:
\begin{enumerate}[label=(\roman*)]
  \item \textbf{Mathematical tractability.} The power-law rate and self-similar daughter law fit naturally within weighted-space well-posedness theory and lead to convenient moment estimates.
  \item \textbf{Size dependence with scale invariance.}  Larger particles break more rapidly, yet the statistics of each split depend only on the ratio $x/y$, in line with experimental evidence.
  \item \textbf{Data-friendly tuning.}  The family is completely characterised by just three parameters, $\alpha_0$, $\lambda$, and $\nu$, which facilitates calibration from limited or noisy data.
\end{enumerate}
This blend of physical realism, analytical simplicity, and ease of calibration explains why the family has become canonical in the fragmentation literature.
\end{remark}

\begin{remark}\label{remark: assumptions}
Assumptions~\ref{hyp:alpha} and \ref{hyp:b}, although standard and widely used in practice, can be relaxed. The analysis below only relies on the following properties:
\begin{enumerate}[label=(\roman*)]
\item the uncontrolled coagulation–fragmentation equation is well-posed in $X_{0,1}$ in the sense of Definition~\ref{def: solution};
\item mass is conserved and each fragmentation event produces only finitely many fragments;
\item  for every $R>0$, $y\in(0,\infty)$ and every Borel set $E\subseteq(0,R)$,
\begin{equation}\label{eqn: holder inequality}
    \int_0^{\min\{y,R\}}
    \mathds{1}_E(x)\,b(x,y)\,\alpha(y)\,dx \leq \rho_R(\lvert E\rvert)(1+y^q),
\end{equation}
where $\lvert E\rvert$ denotes the Lebesgue measure of $E$, $0\leq q<1$ and $\rho_R\colon[0,\infty)\to[0,\infty)$ satisfies $\rho_R(\varepsilon)\to 0$ as $\varepsilon\to 0$;
\item the fragmentation operator generates a $C_0$-semigroup on $X_{0,1}$ (see Proposition~\ref{propo:fragmentation semigroup}).
\end{enumerate}
Condition~(iii) is satisfied by our kernels. Indeed, since $x^\nu$ is nonnegative and nonincreasing on $(0,\infty)$ for $\nu\in(-1,0]$, for every Borel set $E\subset(0,R)$ and every $y>0$ we have
\begin{align*}
    \int_0^{\min\{y,R\}} \mathds{1}_E(x)\,b(x,y)\,\alpha(y)\,dx
    &= \alpha_0(\nu+2)\,y^{\lambda-(\nu+1)}
       \int_0^{\min\{y,R\}} \mathds{1}_E(x)\,x^\nu\,dx \\
    &\le \alpha_0(\nu+2)\,y^{\lambda-(\nu+1)}
       \int_0^{\min\{|E|,y\}} x^\nu\,dx \\
    &= \alpha_0\frac{\nu+2}{\nu+1}\,
       y^{\lambda-(\nu+1)}\min\{|E|,y\}^{\nu+1}.
\end{align*}
Therefore, choosing $q\coloneqq \max\{0,\lambda-(\nu+1)\}$ we have $q\in[0,1)$ and
a simple case distinction on $y\le |E|$ and $y>|E|$ shows that
\begin{equation*}
        y^{\lambda-(\nu+1)}\min\{|E|,y\}^{\nu+1}
    \le C\bigl(|E|^\lambda+|E|^{\nu+1}\bigr)(1+y^q),
\end{equation*}
for some constant $C>0$. Hence condition~(iii) holds with
\begin{equation*}
        \rho_R(\varepsilon)\coloneqq C\bigl(\varepsilon^\lambda+\varepsilon^{\nu+1}\bigr).
\end{equation*}
Assumptions~(i)–(iii) are fulfilled by a fairly large class of kernels and already suffice to prove existence of optimal controls, as well as our main convergence result, provided that the functions $\rho_R$ can be chosen uniformly for all kernels involved. The semigroup property in~(iv) is more restrictive, but it still holds for a broad class of fragmentation kernels, see, for instance, ~\cite[Chapter~5.1.4]{banasiak2019analytic}. We use~(iv) only in the derivation of the Pontryagin minimum principle and in the gradient computations.
\end{remark}

\subsubsection{The coagulation and fragmentation operators}\label{subsubsec: operators}

We present several properties that the coagulation and fragmentation operators possess under \Cref{ass: assumptions}.

First, we split the coagulation operator into its gain and loss parts: $\mathcal{C}f = \mathcal{C}_g f + \mathcal{C}_l f$, where
\begin{align}
    \mathcal{C}_g f(x)
    &\coloneqq \frac12 \int_0^x K(y,\,x-y)\,f(y)\,f(x-y)\,\mathrm{d}y,\\
    \mathcal{C}_l f(x)
    &\coloneqq -\,f(x)\int_0^\infty K(x,y)\,f(y)\,\mathrm{d}y.
\end{align}

The fragmentation operator $\F$ is an unbounded linear operator with domain
\begin{equation*}
    \mathrm{dom}(\mathcal{F})
=
\bigl\{\,f\in X_{0,1} : \alpha\,f \in X_{0,1}\bigr\}.
\end{equation*}
It can also be decomposed into its gain and loss parts, $    \mathcal{F}f(x)=
    \mathcal{F}_g f(x)
    + \mathcal{F}_l f(x)$, by setting
\begin{equation}
    \mathcal{F}_g f(x)=\int_x^{\infty} b(x,y)\,\alpha(y)\,f(y)\,\mathrm{d}y, \hspace{1.5cm}
    \mathcal{F}_l f(x)=-\alpha(x)\,f(x).
\end{equation}
The fragmentation operator generates a semigroup on $X_{0,1}$.  The detailed proof can be found in Theorem 5.1.28 and Corollary 5.1.29 in \cite[Chapter 5]{banasiak2019analytic}. In particular, Theorem 5.1.28 provides more general conditions than \Cref{ass: assumptions} for the existence of the fragmentation semigroup in $X_{0,1}$.

\begin{proposition}\label{propo:fragmentation semigroup}
Under assumptions \ref{hyp:alpha} and \ref{hyp:b}, the fragmentation operator $\mathcal{F}\colon \mathrm{dom}(\mathcal{F}) \subseteq X_{0,1}\;\longrightarrow\;X_{0,1}$
is densely defined, closed and generates a $C_0$-semigroup $\T_\F$ on $X_{0,1}$. In particular there exists $M_\F>0$ and $\omega_\F\in\R$ such that
\begin{equation*}
    \norm{\T_\F(t)}\leq M_\F e^{\omega_\F t}\qquad \forall t\geq 0.
\end{equation*}
\end{proposition}

\subsection{Well-posedness of the controlled equation}

Under \Cref{ass: assumptions} and the boundedness of the control $\control$, the classical well-posedness theory for coagulation–fragmentation equations extends with straightforward modifications to the controlled case. Existence is based on a truncation argument together with weak compactness in $L^1$ that allows for the recovery of the solution of the equation as limit of the solutions of the truncated problem. This technique is standard in the coagulation-fragmentation equations literature, see, for example, \cite[Theorem 8.2.23]{banasiak2019analytic} for a proof of existence in our setting. For uniqueness, we instead refer to \cite{giri2012uniqueness}. Hence, the controlled coagulation–fragmentation equation \eqref{eqn: cf_equation} is well-posed. Moreover, solutions are mass preserving and the growth of the number of particles is controlled.

\begin{theorem}\label{theorem: well posedness}
Given an initial condition $f_{\mathrm{in}}\in X_{0,1}$ and a control $\control\in \U$, there exists a unique solution $\f_\control$ of \eqref{eqn: cf_equation} with such initial condition and control. Moreover, such a solution is mass-conserving, namely
\begin{equation}\label{eqn: conservation of mass}
    M_1\bigl(\f_\control(t)\bigr)=M_1\bigl(f_{\mathrm{in}}\bigr)\qquad \forall t\in [0,T],
\end{equation}
and there exists $\mathbf{C}_{0,1}>0$ depending only on $\alpha$ and $b$ such that for every $t\in [0,T]$
\begin{equation}\label{eqn: M_0 estimate}
    M_0\bigl(\f_\control(t)\bigr)\leq \norm{\f_\control(t)}_{0,1} \leq \norm{f_{\mathrm{in}}}_{0,1} \mathbf{C}_{0,1} e^{\mathbf{C}_{0,1}t}
\end{equation}
\end{theorem}

\section{Optimal control and stability}

In this section, we study the optimal control problem \eqref{eqn: OCP}, establish the existence of optimal controls, and
analyse the stability of minimisers under perturbations of the coagulation
and fragmentation kernels. The key ingredients are a weak-to-weak
continuity result for the control-to-state map and a $\Gamma$-convergence argument.

We recall the cost functional $J\colon  \U \to \R$,
\begin{equation}\label{eq:J-again}
  J(\control) \coloneqq
  \frac{w}{2}\int_0^T\bigl(\control(t)-1\bigr)^2\,dt
  \;+\; \psi\bigl(\f_\control(T)\bigr),
\end{equation}
where $\psi:X_{0,1}\to\R$ and $\f_\control$ is the solution to \eqref{eqn: cf_equation}.

We present the two main results of this section below.

\vspace{0.25cm}

\begin{mdframed}[hidealllines=false,backgroundcolor=blue!5]
\begin{theorem}[Existence of optimal controls]\label{thrm: existence of optimal controls}
Assume that the coefficients $(K,\alpha,b)$ satisfy
Assumptions~\textnormal{H} and that $\psi$ is $X_0$-weakly lower semi-continuous. Then the optimal control problem
\begin{equation*}
    \inf_{\control\in\U} J(\control)
\end{equation*}
admits at least one solution $\control^\star\in\U$.
\end{theorem}
\end{mdframed}

\vspace{0.25cm}

To discuss stability, we consider a sequence of coefficients
$(K_n,\alpha_n,b_n)$ converging to a limit $(K^\star,\alpha^\star,b^\star)$ (as specified in Assumptions C1--C4 in Section~\ref{subsec: continuous dependence}).
Let $J_n$ and $J^\star$ denote the cost functionals associated with the perturbed coefficients and the limit coefficients, respectively, defined analogously to \eqref{eq:J-again}.

\vspace{0.25cm}

\begin{mdframed}[hidealllines=false,backgroundcolor=blue!5]
\begin{theorem}[$\Gamma$-convergence and stability]\label{thrm: gamma convergence}
Under the convergence assumptions of Section~\ref{subsec: continuous dependence}, if $\psi$ is $X_0$-weakly continuous, the sequence $J_n$ $\Gamma$-converges to $J^\star$ with respect to the weak topology of $L^2(0,T)$. Consequently, if $(\control_n)_n$ is a sequence of minimisers for $J_n$, any weak accumulation point $\control^\star$ is a minimiser of $J^\star$.
\end{theorem}
\end{mdframed}

\vspace{0.25cm}

The proofs of these results are organised as follows:

\begin{enumerate}
  \item \textbf{Section \ref{subsec: continuous dependence}: Dependence on the controls.}
  First, we establish the weak-to-weak continuity of the control-to-state map, allowing for varying kernels. The central difficulty is identifying the limit in the trilinear term (control multiplying the coagulation bilinear form) where all factors converge only weakly. We overcome this by repeated applications of Proposition~\ref{propo: egorov} and a tail estimate based on mass conservation.

  \vspace{0.25cm}

  \item \textbf{Section \ref{subsec: existence of optimal controls and stability}: Existence and Stability.}
  We then apply the direct method of the calculus of variations to prove existence of optimal controls. Finally, we prove Theorem~\ref{thrm: gamma convergence} ($\Gamma$-convergence) by exploiting the weak compactness of $\mathcal{U}$ and the continuity results of the previous step.
\end{enumerate}

  \vspace{0.25cm}

\subsection{Continuous dependence} \label{subsec: continuous dependence}

In this part we prove a weak-to-weak continuity result for the dependence of solutions on the controls. We also allow the kernels to vary. More precisely, for every $n\in\N$ we consider a control $\control_n\in\U$, a coagulation kernel $K_n$, and a fragmentation rate and daughter distribution
\begin{equation}\label{eqn: explicit coefficients formula}
    \alpha_n(x)=\alpha_0^n x^{\lambda_n}, 
    \qquad 
    b_n(x,y)=\frac{2 + \nu_n}{y}\biggl(\frac{x}{y}\biggr)^{\nu_n}\mathds{1}_{(0,y)}(x),
\end{equation}
which satisfy \cref{ass: assumptions}, so that there exists a unique corresponding solution $\f_n$. We then study the convergence of the sequence of solutions $(\f_n)_{n\in\N}$.

\newpage

\begin{mdframed}[hidealllines=false,backgroundcolor=blue!5]
\begin{theorem}\label{thm: weak dependence on controls}
Assume that 
\begin{itemize}
    \item[\textbf{(C1)}] $\control_n \rightharpoonup \control^\star$ in $L^2(0,T)$;
    \item[\textbf{(C2)}] $K_n \to K^\star$ uniformly on compact subsets of $(0,\infty)^2$, and there exist $\hat{\mu}>0$ and $\hat{K}_0>0$ such that
    \begin{equation*}
        K_n(x,y),\, K^\star(x,y)\leq \hat{K}_0 (1+x)^{\hat{\mu}}(1+y)^{\hat{\mu}}
        \quad\text{for every }(x,y)\in (0,\infty)^2\text{ and }n\in\N;
    \end{equation*}
    \item[\textbf{(C3)}] $\alpha_0^n\to \alpha_0^\star> 0$ and $\lambda_n\to\lambda^\star\in(0,1)$ with $\lambda^\star>\hat{\mu}$;
    \item[\textbf{(C4)}] $\nu_n\to\nu^\star\in(-1,0]$.
\end{itemize}
Define
\begin{equation*}
    \alpha^\star(x)\coloneqq \alpha_0^\star x^{\lambda^\star},
    \qquad
    b^\star(x,y)\coloneqq \frac{2+\nu^\star}{y}\biggl(\frac{x}{y}\biggr)^{\nu^\star},
\end{equation*}
and let $\f^\star$ denote the unique solution of the controlled coagulation-fragmentation equation associated with the limit data $\control^\star$, $K^\star$, $\alpha^\star$ and $b^\star$, with the same initial condition $f_{\mathrm{in}}$.

Then $\f_n \to \f^\star$ in $C([0,T],X_{0,\mathrm{w}})$,
i.e.\ $\f_n(t)\to \f^\star(t)$ uniformly in time in the weak topology of $X_0$ (see Section \ref{subsec:functional-setting}).
\end{theorem}
\end{mdframed}

\vspace{0.15cm}

\begin{remark}\label{remark: existence}
The proof of Theorem~\ref{thm: weak dependence on controls} also applies when $\alpha^\star$ is approximated by truncations, for instance by taking
\begin{equation*}
    \alpha_n(x)=\alpha^\star(x)\mathds{1}_{[0,n]}(x),
\end{equation*}
since in this case property~(i) of Lemma~\ref{lemma: consequences C} is
trivially satisfied. In the truncated setting, existence and uniqueness of solutions follow from a standard fixed-point argument; see, for instance, \cite[Theorem 3.1]{stewart}. Moreover, the corresponding truncated initial conditions converge strongly to $f_{\mathrm{in}}$ by the dominated convergence theorem. Therefore, by considering a constant control $\control_n\equiv \control$, a straightforward adaptation of the following argument may also be viewed as an existence proof for the full controlled coagulation-fragmentation equation. 
\end{remark}

The proof proceeds in four steps. We begin by establishing some basic consequences of \textbf{(C1)}--\textbf{(C4)}. As a second step, in Lemma~\ref{lemma: compactness}, we show that for each $t\in[0,T]$ the family $\{\f_n(t):n\in\N\}$ is relatively weakly compact in $X_0$. We then prove in Lemma~\ref{lemma: equi-continuity} that $(\f_n)_{n\in\N}$ is equi-continuous in time. These two properties allow us to apply the generalised Ascoli--Arzel\`a theorem in the weak topology of $X_0$ (Theorem~\ref{thrm: Ascoli Arzela}) and extract a convergent subsequence. Finally, we identify every accumulation point as the solution corresponding to the limit control $\control^\star$ and conclude, by uniqueness, that the whole
sequence converges.

For every $n\in\N$ we denote by $\C^n$ and $\F^n$ the coagulation and fragmentation operators associated with $K_n,\alpha_n,b_n$, and by $\C^\star$ and $\F^\star$ the corresponding limit operators associated with $K^\star,\alpha^\star,b^\star$.

The first result establishes uniform bounds on the coefficients as well as convergence properties for the kernels.

\begin{lemma}\label{lemma: consequences C}
Under the assumptions of Theorem~\ref{thm: weak dependence on controls}, the following properties hold:

\begin{enumerate}[label=(\roman*)]
    \item There exist $\hat\lambda\in(0,1)$ and $\hat\alpha_0>0$ such that
    \begin{equation*}
        \alpha_n(x)\le \hat\alpha_0 (1+ x^{\hat\lambda})
        \qquad \forall x\in(0,\infty),\ \forall n\in\mathbb N.
    \end{equation*}
    \item The expected number of fragments is uniformly bounded, i.e.\ there exists $N>0$ such that
    \begin{equation*}
        \sup_{n\in\mathbb N} N_{\nu_n}\le N.
    \end{equation*}
    \item For every $R>0$, there exist $0\leq \hat q<1$ and a function $\hat\rho_R:(0,\infty)\to\mathbb (0,\infty)$ such that $\hat \rho_R(\varepsilon)\to 0$ as $\varepsilon\to 0$ and 
    \begin{equation}\label{eqn: n holder inequality}
    \int_0^{\min\{y,R\} }\mathds{1}_E(x)\,b_n(x,y)\,\alpha_n(y)\,dx \leq \hat\rho_R(\lvert E\rvert)(1+y^{\hat q}),
    \end{equation}

    \item $\alpha_n\to \alpha^\star$ uniformly on compact subsets of $(0,\infty)$.

    \item For every $r>0$ and every $y\in(0,\infty)$,
    \begin{equation*}
        b_n(\cdot,y)\to b^\star(\cdot,y)\qquad\text{in }L^1(0,r).
    \end{equation*}
\end{enumerate}
\end{lemma}

\begin{proof}
Properties (i), (ii), (iii) follow from the convergence of the parameters
$(\alpha_0^n,\lambda_n,\nu_n)$, where for (iii) one mimics the computations of Remark \ref{remark: assumptions}. Properties (iv) and (v)
follow directly from assumptions \textbf{(C1)}--\textbf{(C4)} and the dominated convergence theorem.
\end{proof}

The compactness step, based on the Dunford--Pettis theorem, follows the standard argument in coagulation-fragmentation theory used to prove weak compactness for sequences of solutions associated with truncated kernels; see, for instance, \cite{stewart} or \cite[Chapter~8]{banasiak2019analytic}. We sketch the proof for completeness.

\begin{lemma}\label{lemma: compactness}
The set
\begin{equation*}
    \bigl\{ \f_n(t)\colon n\in\N, t\in [0,T]\bigr\}
\end{equation*}
is relatively weakly compact in $X_0$.
\end{lemma}
\begin{proof}
By the Dunford-Pettis theorem, it is enough to prove that the family $\{f_n(t):n\in\mathbb N,\ t\in[0,T]\}$
is bounded in $X_0=L^1(0,\infty)$, tight, and equi-integrable.

\textbf{Equi-boundedness and tightness:} 
First, simple calculations show that for every $t\in[0,T]$
\begin{equation*}
    \int _0^1 \control_n (t)\C_n \f_n(t,x) dx \leq 0,
\end{equation*}
since controls are positive and a change of order of integration yields
\begin{align*}
    \int_0^1 \C_g^n \f_n(t,x) dx &= \frac{1}{2}\int_0^1 \int_0^{1-x} K(x,y) \f_n(t,x) \f_n(t,y) dy dx\leq  -\frac{1}{2}\int_0^1 \C_l^n \f_n(t,x) dx.
    \end{align*}
Therefore, the growth of the number of particles within $(0,1)$ can be controlled by the fragmentation gain as follows
\begin{align*}
    \int_0^1 \f_n(t,x)dx\leq &\int_0^\infty f_{\mathrm{in}}(x)dx+\int_0^t\int_0^1 \int_x^\infty b_n(x,y) \alpha_n(y) \f_n(\tau, y) dy dx d\tau\\
                        \leq & \int_0^\infty f_{\mathrm{in}}(x)dx + \int_0^t \hat \rho_1(1) \int_0^\infty\f_n(\tau,x) (1+x^{\hat q}) dx d\tau,
\end{align*}
where we exchanged the order of integration and used \eqref{eqn: n holder inequality} with $R=1$ and $E=(0,1)$. As a consequence, due to mass conservation and the inequality $1+x^{\hat q}\leq 2(1+x)$, we have
\begin{align*}
    \int_0^\infty \f_n(t,x) & \leq \int_0^1 \f_n(t,x) dx + \int _1^\infty x\f_n(t,x)dx \leq \norm{f_{\mathrm{in}}}_{0,1}+\int _0^t 2 \hat \rho_1(1) \int_0^\infty \f_n(\tau,x) (1+x) dx d\tau.
\end{align*}
Therefore, Grönwall's inequality and mass conservation yield
\begin{equation}\label{eqn: equiboundedness}
    \norm{\f_n(t)}_{0,1}\leq \mathbf{C}_{eb}\quad \forall n\in\N, \, t\in [0,T]
\end{equation}
for some constant $\mathbf{C}_{eb}>0$ independent of $n$. Tightness in $X_0$ follows from mass conservation since for $R>1$
\begin{equation*}
    \int_R^\infty \f_n(t,x) dx \leq \frac{1}{R}\int_R^\infty x \f_n(t,x) dx \leq \frac{1}{R} \norm{f_{\mathrm{in}}}_1.
\end{equation*}

\textbf{Equi-integrability.}
For $n\in\N$, $t\in[0,T]$, and $\varepsilon>0$, define
\begin{equation*}
    \xi_n(t,\varepsilon):=\sup\left\{\int_E \f_n(t,x)\,dx:\ |E|\le \varepsilon\right\}.
\end{equation*}
We prove that
\begin{equation*}
    \sup_{n\in\N,\; t\in[0,T]}\xi_n(t,\varepsilon)\xrightarrow[\varepsilon\to0]{}0.
\end{equation*}
By tightness it is enough to restrict to Borel sets $E\subset (0,R)$ for $R>0$. Let $E\subset(0,R)$ with $|E|\le \varepsilon$ be fixed. Multiplying \eqref{eqn: solution equation} by $\mathds{1}_E$, integrating over $(0,R)$, and using Fubini's theorem together with the nonnegativity of $f_n$, we obtain
\begin{equation*}
    \int_E \f_n(t,x)\,dx
\leq
\int_E f_{\mathrm{in}}(x)\,dx
+\int_0^t\left(\control_n(\tau)\int_E \mathcal C_g^n \f_n(\tau,x)\,dx
+\int_E \mathcal F_g^n \f_n(\tau,x)\,dx
\right)d\tau .
\end{equation*}
For the coagulation term, a change in the order of integration together with \textbf{(C2)}, gives
\begin{align*}
    \int_E \C_g^n\f_n(\tau, x)dx=\frac{1}{2}\int_0^R \int_0^{R-y} \mathds{1}_{E-y}(x) K_n(y,x) \f_n(\tau,y)\f_n(\tau,x) dx dy
    \leq \frac{1}{2}\hat K_0 (1+R)^{\hat \mu} \mathbf{C}_{eb} \xi_n(\tau,\varepsilon)
\end{align*}
since, for every $y\in(0,R)$, the translated set $(E-y)\cap(0,R-y)$ is contained in $(0,R)$ and has measure not exceeding $\varepsilon$. For the fragmentation term, \eqref{eqn: n holder inequality} yields
\begin{equation*}
    \int_E \mathcal F_g^n \f_n(\tau,x)\,dx
\leq \int_0^\infty \int_0^{\min\{y,R\}} \mathds{1}_E(x) b_n(x,y) \alpha_n(y) dx \f_n(\tau,y) dy \leq 
2 \hat \rho_R(\varepsilon)\mathbf C_{eb}.
\end{equation*}
Therefore,
\begin{equation*}
    \xi_n(t,\varepsilon)
\leq
\eta(\{f_{\mathrm{in}}\},\varepsilon)
+2T\hat\rho_R(\varepsilon)\mathbf C_{eb}
+\frac12 \hat K_0 (1+R)^{\hat \mu} u_{\max}\mathbf C_{eb}
\int_0^t \xi_n(\tau,\varepsilon)\,d\tau,
\end{equation*}
where $\eta(\{f_{\mathrm{in}}\},\varepsilon)$ is the modulus of equi-integrability of $\{f_{\mathrm{in}}\}$. An application of Grönwall's lemma gives
\begin{equation*}
    \xi_n(t,\varepsilon)
\leq
\Bigl(\eta(\{f_{\mathrm{in}}\},\varepsilon)+2T\hat\rho_R(\varepsilon)\mathbf C_{eb}\Bigr)
\exp\!\left(\frac12 \hat K_0 (1+R)^{\hat \mu} u_{\max}\mathbf C_{eb} T\right).
\end{equation*}
Since $f_{\mathrm{in}}\in X_0=L^1(0,\infty)$, its modulus of equi-integrability satisfies $\eta(\{f_{\mathrm{in}}\},\varepsilon)\to 0$ as $\varepsilon\to 0$. Moreover, Lemma~\ref{lemma: consequences C} yields $\hat{\rho}_R(\varepsilon)\to 0$ as $\varepsilon\to 0$. Therefore, the right-hand side converges to $0$ as $\varepsilon\to 0$, as desired.
\end{proof}

Now we show equi-continuity.

\begin{lemma}\label{lemma: equi-continuity}
The sequence $(\f_n)_{n\in\N}$ is strongly equi-continuous in time with values in $X_0$. In particular, there exists a constant $\mathbf C_{ec}>0$, independent of $n$, such that
\begin{equation*}
    \norm{\f_n(t)-\f_n(s)}_0
    \leq \mathbf C_{ec}\,|t-s|
    \quad\forall\, 0\leq s\leq t\leq T,\; n\in\N.
\end{equation*}
\end{lemma}
\begin{proof}
Let $0\leq s<t\leq T$ be fixed. For every $n\in\N$, by the definition of solution and recalling positivity of solutions, controls, and kernels, we obtain
\begin{align*}
    \norm{\f_n(t)-\f_n(s)}_0
    &\leq \int_s^t \int_0^\infty \biggl[
        \control_n(\tau)\,\C_g^n\f_n(\tau,x)
        + \control_n(\tau)\,\bigl\lvert \C_l^n\f_n(\tau,x)\bigr\rvert 
        + \F_g^n\f_n(\tau,x)
        + \bigl\lvert \F_l^n\f_n(\tau,x)\bigr\rvert
    \biggr]\,dx\,d\tau .
\end{align*}

We estimate the four contributions separately. For the coagulation gain term, by straightforward computations, \textbf{(C2)} and Fubini--Tonelli theorem, we have
\begin{align*}
    \int_s^t \int_0^\infty \control_n(\tau)\,\C_g^n\f_n(\tau,x)\,dx\,d\tau
    &\leq \frac{1}{2}u_{\max}\hat K_0 \int_s^t \norm{\f_n(\tau)}_{0,1}^2\,d\tau \leq (t-s)\,\frac{1}{2}u_{\max}\hat K_0\,\mathbf{C}_{eb}^2 .
\end{align*}
Similarly, for the coagulation loss term,
\begin{align*}
    \int_s^t \int_0^\infty \control_n(\tau)\,\bigl\lvert \C_l^n\f_n(\tau,x)\bigr\rvert\,dx\,d\tau
    &\leq u_{\max}\hat K_0 \int_s^t \norm{\f_n(\tau)}_{0,1}^2\,d\tau \leq(t-s)\,u_{\max}\hat K_0\,\mathbf{C}_{eb}^2 .
\end{align*}

The fragmentation loss contribution is handled by the uniform sublinear bound on the fragmentation rates, namely $\alpha_n(x)\leq \hat\alpha_0(1+x^{\hat \lambda})\leq 2\hat\alpha_0(1+x)$, as follows:
\begin{align*}
    \int_s^t \int_0^\infty \bigl\lvert \F_l^n\f_n(\tau,x)\bigr\rvert\,dx\,d\tau\leq 2\hat\alpha_0 \int_s^t \int_0^\infty (1+x)\f_n(\tau,x)\,dx\,d\tau \leq (t-s)\,2\,\hat\alpha_0\,\mathbf{C}_{eb}.
\end{align*}

On the other hand, for the fragmentation gain term, recalling Lemma \ref{lemma: consequences C} (i) and (ii), we have
\begin{align*}
    \int_s^t \int_0^\infty \F_g^n\f_n(\tau,x)\,dx\,d\tau
    &= \int_s^t \int_0^\infty \alpha_n(y)\biggl(\int_0^y b_n(x,y)\,dx\biggr)\f_n(\tau,y)\,dy\,d\tau \\
    &\leq 2 \hat\alpha_0 \int_s^t \int_0^\infty (1+y)\biggl(\int_0^y b_n(x,y)\,dx\biggr)\f_n(\tau,y)\,dy\,d\tau \leq (t-s)\, 2\,\hat\alpha_0 N\,\mathbf{C}_{eb},
\end{align*}
where $N$ is the uniform bound on the number of fragments from Lemma \ref{lemma: consequences C}.

Collecting the previous estimates, we infer that there exists a constant $\mathbf C_{ec}>0$, independent of $n$, such that
\begin{equation*}
    \norm{\f_n(t)-\f_n(s)}_0 \leq \mathbf C_{ec}\,|t-s|,
\end{equation*}
which proves the claim.
\end{proof}

To prove Theorem~\ref{thm: weak dependence on controls} we shall rely on the following consequence
of Egorov's theorem (see \cite[Chapter~7]{banasiak2019analytic}).

\begin{proposition}\label{propo: egorov}
Let $\Omega\subset \R^d$ be Borel, $\mu$ a finite measure, $(f_n)_{n\in\N}\subset L^1(\Omega,\mu)$
and $(g_n)_{n\in\N}\subset L^\infty(\Omega,\mu)$. Assume that
\begin{itemize}
  \item $f_n \rightharpoonup f$ weakly in $L^1(\Omega,\mu)$;
  \item $g_n \to g$ $\mu$-a.e., with $g\in L^\infty(\Omega,\mu)$ and $\sup_n\|g_n\|_{L^\infty}<\infty$.
\end{itemize}
Then $f_n g_n \rightharpoonup f g$ weakly in $L^1(\Omega,\mu)$.
\end{proposition}

For completeness we also recall a weak-topology version of Ascoli--Arzel\`a
\cite[Theorem~7.1.16]{banasiak2019analytic} specialised to $X_0$. 

\begin{theorem}\label{thrm: Ascoli Arzela}
Let $\mathcal E$ be a set of
weakly continuous maps $[0,T]\to X_0$. Then $\mathcal E$ is relatively weakly sequentially compact in
$C([0,T],X_{0,\mathrm{w}})$ if and only if
\begin{itemize}
  \item (\emph{weak equi-continuity}) for every $t_0\in[0,T]$ and every $\phi\in L^\infty(0,\infty)$,
  \begin{equation*}
          \lim_{t\to t_0}\ \sup_{f\in\mathcal E}\lvert\langle\phi,f(t)-f(t_0)\rangle\rvert=0;
  \end{equation*}
  \item there exists a dense set $D\subset[0,T]$ such that, for every $t\in D$,
  the set $\{f(t): f\in\mathcal E\}$ is relatively weakly sequentially compact in $X_0$.
\end{itemize}
\end{theorem}

\medskip

\emph{Proof of Theorem \ref{thm: weak dependence on controls}.}

First, we note that the limit equation is well-posed, as the kernels satisfy \cref{ass: assumptions}.

We know from Lemma \ref{lemma: compactness} that for every $t\in[0,T]$ the set 
\begin{equation*}
    \{ \f_n(t) : n\in\N \}
\end{equation*}
is relatively weakly compact in $X_0$, while Lemma~\ref{lemma: equi-continuity} asserts that the sequence $(\f_n)_{n\in\N}$ is strongly, thus weakly, equi-continuous. Therefore, by Theorem \ref{thrm: Ascoli Arzela} there exists a subsequence (not relabelled) and a function $\f^\star\in C([0,T],X_{0,w})$ such that 
\begin{equation*}
    \f_n(t)\rightharpoonup \f^\star(t)\quad\text{in }X_0\text{ for all }t\in[0,T],
\end{equation*}
and the convergence is uniform in time in the weak topology of $X_0$. 

We claim that the uniform bound \eqref{eqn: equiboundedness} is inherited by the
limit. Fix $t\in[0,T]$. For $r>0$, define $\phi_r(x)\coloneqq \min\{1+x,r\}$. Since $\phi_r\in L^\infty(0,\infty)$ and $\f_n(t)\rightharpoonup \f^\star(t)$
weakly in $X_0$, we have
\begin{equation*}
    \int_0^\infty \f_n(t,x)\phi_r(x)\,dx
    \longrightarrow
    \int_0^\infty \f^\star(t,x)\phi_r(x)\,dx.
\end{equation*}
Moreover, $0\leq \phi_r(x)\le 1+x$, hence
\begin{equation*}
    \int_0^\infty \f^\star(t,x)\phi_r(x)\,dx
    \leq
    \limsup_{n\to\infty}\int_0^\infty \f_n(t,x)(1+x)\,dx
    \leq \mathbf C_{eb}.
\end{equation*}
Since $\phi_r\uparrow 1+x$ as $r\to\infty$ and $\f^\star(t,\cdot)\ge 0$ a.e.,
the monotone convergence theorem yields
\begin{equation}\label{eqn: limit equiboundedness}
    \int_0^\infty \f^\star(t,x)(1+x)\,dx=\norm{\f^\star(t)}_{0,1}\leq \mathbf C_{eb}
\end{equation}
Thus, \eqref{eqn: equiboundedness} also holds for $\f^\star$.

We now show that $\f^\star$ is a solution of \eqref{eqn: cf_equation} with the limit control $\control^\star$ and initial condition $f_{\mathrm{in}}$.

For every $n\in\N$, $t\in[0,T]$ and a.e.\ $x\in(0,\infty)$ it holds
\begin{equation*}
        \f_n(t,x) = f_{\mathrm{in}}(x)
    + \int_0^t \biggl[
        \control_n(\tau)\,\C^n\bigl(\f_n(\tau)\bigr)(x)
        + \F^n\f_n(\tau)(x)
    \biggr]\,d\tau.
\end{equation*}
We know that, for each fixed $t\in[0,T]$, the left-hand side converges weakly to $\f^\star(t)$ in $L^1(0,\infty)$ and hence in $L^1(0,R)$ for every $R>0$. We now show that, for every $R>0$ and every $t\in[0,T]$, also the sequence of right-hand sides
\begin{equation*}
        x \mapsto f_{\mathrm{in}}(x)
    + \int_0^t \biggl[
        \control_n(\tau)\,\C^n\bigl(\f_n(\tau)\bigr)(x)
        + \F^n\f_n(\tau)(x)
    \biggr]\,d\tau
\end{equation*}
converges weakly in $L^1(0,R)$, and that its limit coincides with
\begin{equation*}
        x \mapsto f_{\mathrm{in}}(x)
    + \int_0^t \biggl[
        \control^\star(\tau)\,\C^\star\bigl(\f^\star(\tau)\bigr)(x)
        + \F^\star \f^\star(\tau)(x)
    \biggr]\,d\tau.
\end{equation*}

To this end we fix $R>0$ and a test function $\varphi\in L^\infty(0,R)$, and prove the convergence of the corresponding duality pairings for each operator $\C^n_l$, $\C^n_g$, $\F^n_l$ and $\F^n_g$ independently.

\smallskip

\textbf{Coagulation loss.}
We want to prove that
\begin{equation*}
    \lim_{n\to+\infty}
    \int_0^R \varphi(x)\int_0^t \control_n(\tau)\,\C^n_l\bigl(\f_n(\tau)\bigr)(x)\,d\tau\,dx
    =
    \int_0^R \varphi(x)\int_0^t \control^\star(\tau)\,\C^\star_l\bigl(\f^\star(\tau)\bigr)(x)\,d\tau\,dx.
\end{equation*}
First,
we claim that, for every $\tau\in[0,T]$ and $x\in[0,R]$,
\begin{equation}\label{eqn: coagulation loss pw convergence}
        \int_0^\infty K_n(x,y)\,\f_n(\tau,y)\,dy
    \longrightarrow
    \int_0^\infty K^\star(x,y)\,\f^\star(\tau,y)\,dy
    \quad\text{as }n\to\infty.
\end{equation}
Indeed, for every $r>0$ we have
\begin{align*}
    \biggl\lvert
        \int_0^\infty
        \bigl( K_n(x,y)\f_n(\tau,y) - K^\star(x,y)\f^\star(\tau,y) \bigr)\,dy
    \biggr\rvert
    &\leq \biggl\lvert
    \int_0^r
         K_n(x,y)\f_n(\tau,y) - K^\star(x,y)\f^\star(\tau,y) \,dy\biggr\rvert\\
    &\quad
    + \int_r^\infty
        \bigl\lvert K_n(x,y)\f_n(\tau,y) - K^\star(x,y)\f^\star(\tau,y) \bigr\rvert\,dy.
\end{align*}
The first term on the right-hand side goes to zero, as $n\to\infty$, due to the weak convergence of $\f_n(\tau)$ to $\f^\star(\tau)$ in $X_0$ and the uniform convergence of the coagulation kernels on $[0,R]\times[0,r]$. For the second term, using the uniform growth bound on $K_n$ and $K^\star$, we obtain the uniform estimate
\begin{align*}
    \int_r^\infty
        \bigl\lvert K_n(x,y)\f_n(\tau,y) - K^\star(x,y)\f^\star(\tau,y) \bigr\rvert\,dy
    &\leq
    2\,\hat{K}_0\,\mathbf{C}_{eb}\,\frac{1+R}{(1+r)^{1-\hat{\mu}}},
\end{align*}
which can be made arbitrarily small by choosing $r$ large enough, uniformly in $n$, $\tau$ and $x\in[0,R]$. Hence, \eqref{eqn: coagulation loss pw convergence} is proved.

Next, the sequence of functions
\begin{equation*}
    x \mapsto \int_0^\infty K_n(x,y)\,\f_n(\tau,y)\,dy
\end{equation*}
and its limit are uniformly bounded on $[0,R]$, due to the sublinear growth of the coagulation kernels and \eqref{eqn: equiboundedness}. Combining the weak convergence of $\f_n(\tau)$ to $\f^\star(\tau)$ and the pointwise convergence we just proved, we can apply Proposition~\ref{propo: egorov} and obtain the following pointwise convergence in $\tau$:
\begin{equation*}
        \lim_{n\to\infty}
    \int_0^R \varphi(x)\,\f_n(\tau,x)\int_0^\infty K_n(x,y)\,\f_n(\tau,y)\,dy\,dx
    =
    \int_0^R \varphi(x)\,\f^\star(\tau,x)\int_0^\infty K^\star(x,y)\,\f^\star(\tau,y)\,dy\,dx.
\end{equation*}
Using again the uniform bounds on the kernels, the boundedness of $\varphi$ and \eqref{eqn: equiboundedness}, it is easy to see that the sequence of functions of $\tau$ on the left-hand side and its limit are uniformly bounded on $[0,T]$. Since $\control_n$ converges weakly in $L^2(0,T)$ and hence in $L^1(0,T)$, we can now further apply Proposition~\ref{propo: egorov} to conclude that the sequence obtained by multiplying these functions by $\control_n(\tau)$ converges weakly in $L^1(0,T)$ to the corresponding limit with $\control^\star(\tau)$.

Recalling the definition of $\C_l$ and testing in time against $\chi_{[0,t]}\in L^\infty(0,T)$, we obtain
\begin{equation*}
        \lim_{n\to+\infty}\int_0^t \control_n(\tau)\int_0^R\varphi(x)\,\C^n_l\bigl(\f_n(\tau)\bigr)(x)\,dx\,d\tau
    =
    \int_0^t \control^\star(\tau)\int_0^R\varphi(x)\,\C^\star_l\bigl(\f^\star(\tau)\bigr)(x)\,dx\,d\tau.
\end{equation*}
Exchanging the order of integration by means of the Fubini--Tonelli theorem gives the desired weak convergence for the coagulation loss term.

\smallskip

\textbf{Coagulation gain.}
We now need to prove that
\begin{equation}\label{eqn: coagulation gain convergence}
    \lim_{n\to+\infty}
    \int_0^R\varphi(x)\int_0^t \control_n(\tau)\,\C^n_g\bigl(\f_n(\tau)\bigr)(x)\,d\tau\,dx
    =
    \int_0^R\varphi(x)\int_0^t \control^\star(\tau)\,\C^\star_g\bigl(\f^\star(\tau)\bigr)(x)\,d\tau\,dx.
\end{equation}
By changing the order of integration and applying the usual change of variables in the gain term, we obtain that for every $\tau\in[0,T]$
\begin{equation*}
        \int_0^R \varphi(x)\,\C^n_g\bigl(\f_n(\tau)\bigr)(x)\,dx
    =
    \frac{1}{2}\int_0^R \int_0^{R-y}
        \varphi(x+y)\,K_n(x,y)\,\f_n(\tau,x)\,\f_n(\tau,y)\,dx\,dy.
\end{equation*}
For each fixed $y\in[0,R]$ the map $x\mapsto\varphi(x+y)K_n(x,y)$ converges uniformly on $[0,R-y]$ to $x\mapsto\varphi(x+y)K^\star(x,y)$. Therefore, the same argument based on Proposition~\ref{propo: egorov} that we used for the coagulation loss term applies here as well, and yields \eqref{eqn: coagulation gain convergence}. In this case no tail-control argument is required, since the domain of integration in $x$ and $y$ is bounded.

\smallskip

\textbf{Fragmentation loss.}
We now show that
\begin{equation}\label{eqn: fragmentation loss convergence}
    \lim_{n\to+\infty}\int_0^R \varphi(x)\int_0^t\alpha_n(x)\,\f_n(\tau,x)\,d\tau\,dx
    =
    \int_0^R \varphi(x)\int_0^t\alpha^\star(x)\,\f^\star(\tau,x)\,d\tau\,dx.
\end{equation}
In this case the control does not appear, and the argument is simpler. After exchanging the order of integration in $x$ and $\tau$, it is enough to pass to the limit inside the integral in time.

For every fixed $\tau\in[0,T]$,
\begin{equation*}
        \lim_{n\to\infty}
    \int_0^R \varphi(x)\,\alpha_n(x)\,\f_n(\tau,x)\,dx
    =
    \int_0^R \varphi(x)\,\alpha^\star(x)\,\f^\star(\tau,x)\,dx
\end{equation*}
follows from the uniform convergence of $\alpha_n$ to $\alpha^\star$ on $[0,R]$ and the weak convergence of $\f_n(\tau)$ to $\f^\star(\tau)$ in $X_0$. Moreover, the uniform bound on $\alpha_n$ yields
\begin{equation*}
        \biggl\lvert \int_0^R \varphi(x)\,\alpha_n(x)\,\f_n(\tau,x)\,dx \biggr\rvert
    \leq \|\varphi\|_\infty\,\hat{\alpha}_0\,R^{\hat{\lambda}}\,\mathbf{C}_{eb},
\end{equation*}
so the dominated convergence theorem applies in $\tau$ and gives \eqref{eqn: fragmentation loss convergence}.

\smallskip

\textbf{Fragmentation gain.}
To prove the convergence of the last duality pairing limit we introduce the functions
\begin{equation*}
        \xi_n(y)\coloneqq\int_0^{\min\{y,R\}}\varphi(x)\,b_n(x,y)\,\alpha_n(y)\,dx,
    \qquad
    \xi^\star(y)\coloneqq\int_0^{\min\{y,R\}}\varphi(x)\,b^\star(x,y)\,\alpha^\star(y)\,dx,
\end{equation*}
defined for every $y\in(0,\infty)$. Due to (v) in Lemma \ref{lemma: consequences C} we have $\xi_n(y)\to\xi^\star(y)$ pointwise for all $y>0$,
\begin{equation}\label{eqn: xi bound}
        |\xi_n(y)|,\;|\xi^\star(y)| \leq \|\varphi\|_\infty\,\hat{\rho}_R(R)(1+y^{\hat q})
    \qquad \text{for all }y>0,
\end{equation}
with $\hat q<1$.

After a standard change in the order of integration, our goal reduces to proving
\begin{equation}\label{eqn:fragmentation-gain-goal}
    \lim_{n\to\infty}\int_0^t\int_0^\infty \f_n(\tau,y)\,\xi_n(y)\,dy\,d\tau
    = \int_0^t\int_0^\infty \f^\star(\tau,y)\,\xi^\star(y)\,dy\,d\tau.
\end{equation}

We first observe that
\begin{equation}\label{eqn:frag-gain-1}
    \lim_{n\to\infty}\int_0^t\int_0^\infty
        \bigl(\f_n(\tau,y)-\f^\star(\tau,y)\bigr)\,\xi_n(y)\,dy\,d\tau = 0.
\end{equation}
For each fixed $\tau\in[0,T]$, the inner integrals over $(0,\infty)$ converge to zero as we can split them over $(0,r)$ and $(r,\infty)$ with $r>0$. The integrals over $(0,r)$ converge by Proposition~\ref{propo: egorov}, which combines the weak convergence of $\f_n(\tau)$ to $\f^\star(\tau)$, the pointwise convergence $\xi_n\to\xi^\star$ and the uniform bound on $(0,r)$ due to \eqref{eqn: xi bound}. On the other hand the integral over $(r,\infty)$ is controlled by \eqref{eqn: xi bound} through the equi-boundedness \eqref{eqn: equiboundedness}.  The same bounds, together with \eqref{eqn: equiboundedness}, provide a uniform $L^1(0,T)$ bound, so the dominated convergence theorem applies in time and yields \eqref{eqn:frag-gain-1}.

We now show that
\begin{equation}\label{eqn:frag-gain-2}
    \lim_{n\to\infty}\int_0^t\int_0^\infty
        \f^\star(\tau,y)\,\bigl(\xi_n(y)-\xi^\star(y)\bigr)\,dy\,d\tau = 0.
\end{equation}
Once again, the integrals are uniformly bounded in time, so it is enough to prove the pointwise convergence for each fixed $\tau\in[0,T]$. Let $r>1$. Then
\begin{equation*}
        \biggl|\int_0^\infty \f^\star(\tau,y)\,\bigl(\xi_n(y)-\xi^\star(y)\bigr)\,dy\biggr|
    \leq
    \int_0^r \f^\star(\tau,y)\,\bigl|\xi_n(y)-\xi^\star(y)\bigr|\,dy
    + \int_r^\infty \f^\star(\tau,y)\,\bigl|\xi_n(y)-\xi^\star(y)\bigr|\,dy.
\end{equation*}
By the dominated convergence theorem, the first term on the right-hand side converges to zero as $n\to\infty$: we are working on the bounded set $(0,r)$, where $\xi_n\to\xi^\star$ almost everywhere and the sequence $(\xi_n)_n$ is uniformly bounded.

On the other hand, on $(r,\infty)$,  the bounds \eqref{eqn: limit equiboundedness} and \eqref{eqn: xi bound} yield the uniform tail estimate
\begin{equation*}
        \int_r^\infty \f^\star(\tau,y)\,\bigl|\xi_n(y)-\xi^\star(y)\bigr|\,dy
    \leq
    4\,\|\varphi\|_\infty\,\hat \rho_R(R)\int_r^\infty (1+y)^{\hat q}\f^\star(\tau,y)\,dy
    \leq
    \frac{4\,\mathbf{C}_{eb}\,\|\varphi\|_\infty\,\hat \rho_R(R)}{(1+r)^{1-\hat q}},
\end{equation*}
so the right-hand side can be made arbitrarily small by choosing $r$ large enough, uniformly in $n$ and $\tau$. This proves \eqref{eqn:frag-gain-2}. Combining \eqref{eqn:frag-gain-1} and \eqref{eqn:frag-gain-2} we obtain \eqref{eqn:fragmentation-gain-goal}, as desired.

\smallskip

We have thus proved the desired weak convergence in $L^1(0,R)$ for every $R>0$. Therefore, by uniqueness of the weak limit, we obtain that for a.e.\ $x\in(0,R)$ and all $t\in[0,T]$ it holds
\begin{equation*}
        \f^\star(t,x) = f_{\mathrm{in}}(x)
    + \int_0^t \biggl[
        \control^\star(\tau)\,\C^\star\bigl(\f^\star(\tau)\bigr)(x)
        + \F^\star\f^\star(\tau)(x)
    \biggr]\,d\tau.
\end{equation*}
Given the arbitrariness of $R$, the previous equality holds for almost every $x\in(0,\infty)$. This shows that $\f^\star$ is a solution of the limit coagulation-fragmentation equation with control $\control^\star$ and initial condition $f_{\mathrm{in}}$. Since this Cauchy problem admits a unique solution, every subsequence of $(\f_n)_{n\in\N}$ admits a subsubsequence converging to $\f^\star$. Therefore, the whole sequence $(\f_n)_{n\in\N}$ converges to $\f^\star$ weakly in $X_0$, uniformly in time, and the proof is complete.

\qed

\subsection{Proofs of existence and stability}\label{subsec: existence of optimal controls and stability}

In this subsection, we provide the proofs for the main results stated in the introduction of this section, relying on the continuity properties established in Theorem~\ref{thm: weak dependence on controls}.

\begin{proof}[Proof of Theorem~\ref{thrm: existence of optimal controls}]
The set of admissible controls $\U$ is closed, convex, and bounded in $L^2(0,T)$, and is therefore weakly compact. 
Theorem~\ref{thm: weak dependence on controls} ensures that the control-to-state map is continuous with respect to the weak topology of $\U$; specifically, if $\control_n \rightharpoonup \control^\star$ in $L^2(0,T)$, then $f_{\control_n}(T) \rightharpoonup f_{\control^\star}(T)$ in $X_0$. 
Since $\psi$ is lower semi-continuous with respect to the weak topology in $X_0$ and the quadratic running cost is convex and continuous (hence weakly lower semi-continuous), the functional
\begin{equation*}
    \control \;\mapsto\; \frac{w}{2}\int_0^T (\control(t)-1)^2\,\mathrm{d}t + \psi \bigl(\f_\control(T)\bigr)
\end{equation*}
is sequentially weakly lower semi-continuous on the weakly compact set $\U$. The existence of a global minimiser follows by the direct method of the calculus of variations.
\end{proof}

Next, we address the stability of optimal controls. We recall the functionals $J_n$ and $J^\star$ defined at the beginning of the section associated with the converging coefficients. 

\begin{proof}[Proof of Theorem~\ref{thrm: gamma convergence}]
We verify the two conditions for $\Gamma$-convergence in the weak topology of $L^2(0,T)$.

\vspace{0.1cm}

\textbf{$\boldsymbol{\Gamma}$-liminf inequality.}
Let $\control_n \rightharpoonup \control$ in $\U$. The quadratic term is convex and continuous, so by weak lower semi-continuity:
\begin{equation*}
    \frac{w}{2}\int_0^T (\control(t)-1)^2\,\mathrm{d}t
    \;\leq\;
    \liminf_{n\to\infty} \frac{w}{2}\int_0^T (\control_n(t)-1)^2\,d t.
\end{equation*}
Regarding the terminal cost, Theorem~\ref{thm: weak dependence on controls} implies that $\f^n_{\control_n}(T) \rightharpoonup \f^\star_\control(T)$ in $X_0$. Since $\psi$ is weakly continuous, we have
\begin{equation*}
    \psi\bigl(\f^\star_\control(T)\bigr)
    \;\leq\;
    \liminf_{n\to\infty} \psi\bigl(\f^n_{\control_n}(T)\bigr).
\end{equation*}
Summing these inequalities yields $J^\star(\control) \;\leq\; \liminf_{n\to\infty} J_n(\control_n)$.

\vspace{0.1cm}

\textbf{$\boldsymbol{\Gamma}$-limsup inequality.}
Fix any $\control \in\U$. We choose the constant recovery sequence $\control_n \coloneqq \control$ for all $n$. Applying Theorem~\ref{thm: weak dependence on controls} again (with constant controls), we obtain
\begin{equation*}
    \f^n_{\control}(T) \rightharpoonup \f^\star_{\control}(T)\quad\text{in }X_0.
\end{equation*}
$\psi$ is weakly continuous in $X_0$, therefore,
\begin{equation*}
        \psi\bigl(\f^n_{u}(T)\bigr) \to \psi\bigl(\f^\star_{\control}(T)\bigr).
\end{equation*}
Since the control is fixed, the running cost is constant for all $n$. Thus, $J_n(\control_n) \to J^\star(\control)$, establishing the $\Gamma$-limsup inequality.

\vspace{0.1cm}

\textbf{Convergence of minimisers.}
Since $\U$ is bounded, closed and convex in $L^2(0,T)$, it is weakly compact and thus the sequence of functionals is equi-coercive. By the fundamental theorem of $\Gamma$-convergence, any cluster point of a sequence of minimisers of $J_n$ is a minimiser of $J^\star$.
\end{proof}

\begin{remark}
This theorem provides a justification for the use of truncated kernels in the numerical optimisation procedure. While the gradient computations in Section~\ref{sec: gradients and optimality conditions} require bounded kernels to ensure Fréchet differentiability of the coagulation operator, Theorem~\ref{thrm: gamma convergence} ensures that the optimal controls obtained for these bounded approximations converge to an optimiser of the original physical problem with unbounded kernels.
\end{remark}

\section{First-order analysis and necessary optimality conditions}\label{sec: gradients and optimality conditions}

In this section we derive an adjoint-based representation of the first variation of the cost functional with respect to the control, together with a Pontryagin-type minimum principle for (locally) optimal control. We work in the weighted state space $X_{0,1}$ and in its dual space
$X_{\infty,1} \coloneqq X_{0,1}^*$, endowed with the duality pairing
$\langle\cdot,\cdot\rangle$; see
Subsection~\ref{subsec: adjoint space} for further details.  Throughout this section we assume that
the coagulation kernel $K$ is bounded:

\vspace{0.15cm}

\begin{itemize}
    \item[\textbf{(H2')}] $K\in L^\infty\bigl((0,\infty)\times(0,\infty)\bigr)$.
\end{itemize}

\vspace{0.15cm}

\noindent This implies that the coagulation operator $\C$ is differentiable on $X_{0,1}$; see Proposition~\ref{propo: coagulation operator}. Note that this is a strengthening of \textbf{(H2)} as it requires $\mu=0$.

Recall that, for $\control\in\U$, $\f_\control$ denotes the solution of
\eqref{eqn: cf_equation} with initial datum $f_{\mathrm{in}}$ and control
$\control$, and that the cost functional is
\begin{equation}\label{eqn: J-again}
  J(\control)=
  \frac{w}{2}\int_0^T\bigl(\control(t)-1\bigr)^2\,\mathrm{d}t
  + \psi\bigl(\f_\control(T)\bigr),
\end{equation}
where $\psi\colon X_{0,1}\to\R$ is Fréchet differentiable. Throughout this section, the initial datum $f_{\mathrm{in}}$ is fixed once and
for all. Accordingly, the dependence of constants on $f_{\mathrm{in}}$ will be
omitted.

The section contains two main results: a first-order expansion of the cost functional along admissible directions, together
with Lipschitz continuity of this gradient, and a Pontryagin minimum
principle. The presentation relies on the equivalence between the notion of
solution introduced in Definition~\ref{def: solution} and the mild semigroup
formulation proved in Subsection~\ref{subsec: mild}.

We denote by $\langle\cdot,\cdot\rangle_2$ and $\|\cdot\|_2$ the scalar
product and norm on $L^2(0,T)$.

\vspace{0.25cm}

\begin{mdframed}[hidealllines=false,backgroundcolor=blue!5]
\begin{theorem}[First-order expansion of the cost functional]\label{thrm: gradients}
Let $\psi\colon X_{0,1}\to\mathbb{R}$ be Fréchet differentiable with
differential $D\psi(f)\in X_{\infty,1}$ for every $f\in X_{0,1}$.
Then, for every $\control^\star\in\U$ and every direction $\delta \control\in L^2(0,T)$
such that $\control^\star+\varepsilon\delta \control\in \U$ for all sufficiently small $\varepsilon>0$, we have
\begin{equation}\label{eqn: directional derivative J}
    J\bigl(\control^\star+\varepsilon\delta \control\bigr)
    = J\bigl(\control^\star\bigr)
      +\varepsilon\bigl\langle\nabla_u J(\control^\star), \delta \control\bigr\rangle_2
      + o (\varepsilon),
\end{equation}
where $\nabla_u J(\control^\star)\in L^2(0,T)$ is given by
\begin{equation}\label{eq:grad-explicit}
    \nabla_u J(\control^\star)(t)
    = w \bigl(\control^\star(t)-1\bigr)
      + \big\langle \bfphi_{\control^\star}(t), \mathcal{C}\f_{\control^\star}(t)\big\rangle,
    \qquad t\in[0,T].
\end{equation}
Here $\bfphi_{\control^\star}$ is the unique solution of the backward adjoint
Cauchy problem
\begin{equation}\label{eqn: adjoint Cauchy problem}
\begin{cases}
\partial_t \varphi(t)
   =-\F^*\varphi(t)-\control^\star(t)\, D\C[\f_{\control^\star}(t)]^*\varphi(t),\\[2pt]
\varphi(T)= D\psi\bigl(\f_{\control^\star}(T)\bigr).
\end{cases}
\end{equation}
Moreover, if $D\psi$ is Lipschitz on bounded subsets of $X_{0,1}$,
then the map
\begin{equation*}
    \U\ni \control\longmapsto \nabla_uJ(\control)\in L^2(0,T)
\end{equation*}
is Lipschitz continuous.
\end{theorem}
\end{mdframed}

\vspace{0.15cm}

For later use, we introduce the reduced Hamiltonian
\begin{equation}\label{eq:hamiltonian}
    H\colon X_{0,1}\times X_{\infty,1}\times \R \to \R \hspace{1.5cm} H(f,\varphi,\omega)
    \coloneqq \frac{w}{2}(\omega-1)^2+\omega\langle \varphi, \C f\rangle.
\end{equation}

\vspace{0.15cm}

\begin{mdframed}[hidealllines=false,backgroundcolor=blue!5]
\begin{theorem}[Pontryagin minimum principle]\label{thm: PMP}
Assume that the terminal cost $\psi$ is Fréchet differentiable. Let $\control^\star\in\U$ be a locally optimal control in the strong  $L^2(0,T)$ topology and let $\f^\star=f_{\control^\star}$ be the corresponding optimal trajectory.
Let $\bfphi^\star=\bfphi_{\control^\star}$ denote the solution of the adjoint
equation~\eqref{eqn: adjoint Cauchy problem} associated with
$(\f^\star,\control^\star)$. Then, for almost every $t\in[0,T]$,
\begin{equation*}
    H\bigl(\f^\star(t), \bfphi^\star(t), \control^\star(t)\bigr)
    = \min_{\omega\in [u_{\min}, u_{\max}]} H\bigl(\f^\star(t), \bfphi^\star(t), \omega\bigr).
\end{equation*}
In particular, if $P_{[u_{\min},u_{\max}]}$ denotes the pointwise
projection from $\R$ onto $[u_{\min}, u_{\max}]$ and $w>0$, we have
\begin{equation}\label{eqn: feedback}
  \control^\star(t)
  \;=\;
  P_{[u_{\min},u_{\max}]}
  \Bigl( 1 - \tfrac{1}{w}\,\bigl\langle \bfphi^\star(t), \mathcal{C}\f^\star(t)\bigr\rangle \Bigr),
  \qquad\text{for a.e. }t\in[0,T].
\end{equation}
\end{theorem}
\end{mdframed}

\vspace{0.15cm}

\begin{remark}
The terminal cost presented in the introduction 
\begin{equation*}
    \psi(f)=\pm\int_{x_{\min}}^{x_{\max}} f(x)\,dx
\end{equation*}
satisfies the assumptions of the previous theorems, indeed it is linear and continuous so that its differential is independent of $f$ and given by
\begin{equation*}
    D \psi (f) (x)= \pm\mathds{1}_{[x_{\min}, x_{\max}]}(x).
\end{equation*}
\end{remark}

The remainder of this section is devoted to the construction of the adjoint
equation in the dual space $X_{\infty,1}$ and to the proofs of
Theorems~\ref{thrm: gradients} and~\ref{thm: PMP}.

\subsection{Adjoint space and adjoint operators}\label{subsec: adjoint space}

Our adjoint variables live in the dual space
\begin{equation}\label{eqn: dual space}
  X_{\infty,1}\coloneqq X_{0,1}^*
  =\Bigl\{\varphi:(0,\infty)\to\R \text{ measurable} :
    \|\varphi\|_{\infty,1}
    \coloneqq \operatorname*{ess\,sup}_{x>0}
    \frac{|\varphi(x)|}{1+x} < \infty \Bigr\},
\end{equation}
equipped with the duality pairing 
\begin{equation}\label{eqn: duality pairing}
  \langle \varphi, f\rangle
  \coloneqq \int_0^\infty \varphi(x)\,f(x)\,\mathrm{d}x,
  \qquad
  \varphi\in X_{\infty,1},\; f\in X_{0,1}.
\end{equation}

The boundedness of $K$ implies that the coagulation operator
$\C\colon X_{0,1}\to X_{0,1}$ is continuously Fréchet differentiable (see
\cite[Theorem 8.1.1]{banasiak2019analytic}); for completeness we recall the
explicit formula and a local Lipschitz estimate.

\begin{proposition}\label{propo: coagulation operator}
The coagulation operator $\mathcal{C}\colon X_{0,1}\to X_{0,1}$ is
 of class $C^1$. For every $f\in X_{0,1}$, its Fréchet derivative $D\mathcal C[f]\in \mathcal L(X_{0,1})$
is given, for $g\in X_{0,1}$, by
\begin{align*}
    D\mathcal{C}[f] g (x)
    &= \int_0^x K(x-y, y)\, f(x-y)\, g(y)\,\mathrm{d}y \\
    &\quad - \int_0^\infty K(x,y)\, f(x)\,g(y)\,\mathrm{d}y
           - \int_0^\infty K(x,y)\, f(y)\,g(x)\,\mathrm{d}y .
\end{align*}
In particular, the map $ f\mapsto D\mathcal C[f]$
is linear and bounded from $X_{0,1}$ into $\mathcal L(X_{0,1})$ with
\begin{equation*}
    \|D\mathcal C[f]\|_{\mathcal L(X_{0,1})}
    \leq 4\|K\|_\infty \|f\|_{0,1},
    \qquad \forall f\in X_{0,1}.
\end{equation*}
Moreover, $\C$ is locally Lipschitz on bounded sets: if
$\norm{f_1}_{0,1}, \norm{f_2}_{0,1}\leq C$, then
\begin{equation*}
    \norm{\C f_2-\C f_1}_{0,1}\leq 4C\,\|K\|_\infty \,\norm{f_2-f_1}_{0,1}.
\end{equation*}
\end{proposition}

We next record the expressions of the adjoints of the linearised coagulation
operator and of the fragmentation operator.

\begin{lemma}[Adjoint operators]\label{lemma: adjoint operators}
Let $f\in X_{0,1}$. Then:
\begin{enumerate}[label=(\roman*)]
  \item The adjoint
        $D\mathcal{C}[f]^*: X_{\infty,1}\to X_{\infty,1}$ is given by
        \begin{align*}
          D\mathcal{C}[f]^* \varphi (x)
          &= \int_0^\infty K(y,x)\,f(y)\,\varphi(x+y)\,\mathrm{d}y \\
          &\quad - \int_0^\infty K(y,x)\,f(y)\,\varphi(y)\,\mathrm{d}y
                - \int_0^\infty K(x,y)\,f(y)\,\varphi(x)\,\mathrm{d}y.
        \end{align*}
        Moreover, there exists a constant $C>0$ (depending only on $K$)
        such that
        \begin{equation*}
            \norm{D\C[f]^*}
            \;\leq\; C\,\norm{K}_\infty\,\norm{f}_{0,1}
            \qquad\forall\,f\in X_{0,1}.
        \end{equation*}
  \item The adjoint fragmentation operator $\F^*\colon         \mathrm{dom}(\F^*) = L^\infty(0,\infty)
        \subseteq X_{\infty,1} \longrightarrow X_{\infty,1}$ is given by
        \begin{equation*}
            \F^*\varphi(x)
            = -\alpha(x)\,\varphi(x)
            + \int_0^x \alpha(x)\,b(y,x)\,\varphi(y)\,dy.
        \end{equation*}
\end{enumerate}
\end{lemma}

\begin{remark}\label{remark: adjoint fragmentation}
The space $X_{0,1}$ is not reflexive, and its dual
$X_{\infty,1}=X_{0,1}^*$ is strictly larger than $L^\infty(0,\infty)$.
Although $\F$ is densely defined on $X_{0,1}$, the natural
domain of its adjoint is $\mathrm{dom}(\F^*) = L^\infty(0,\infty)$, which
is only weak-$*$ dense in $X_{\infty,1}$ and not dense in the norm
topology. As a consequence, the adjoint (backward) equation will be formulated in the weak-$*$ sense, that is, we
characterise solutions by duality against every $f\in X_{0,1}$.
\end{remark}

\subsection{The mild formulation and the linearised dynamics}\label{subsec: mild}

In this part we deal with the linearisation of the coagulation-fragmentation equation. To do so, the notion of mild semigroup solution will be instrumental.

\begin{definition}\label{def: mild solution}
Given a nonnegative $f_{\mathrm{in}}\in X_{0,1}$ and a control $\control\in \U$, we say that $\f\colon [0,T]\to X_{0,1}$ is a \emph{mild solution} of \eqref{eqn: cf_equation} with such data if
\begin{equation*}
    \f(t)=\T_\F(t) f_{\mathrm{in}} +\int_0^t \T_\F(t-\tau)\,\control(\tau)\,\C \f(\tau)\, d\tau,
\end{equation*}
where $(\T_\F(t))_{t\ge 0}$ is the fragmentation semigroup on $X_{0,1}$.
\end{definition}

We now show that every solution of \eqref{eqn: cf_equation} according to Definition \ref{def: solution} is indeed a mild solution.

\begin{proposition}\label{prop: stewart implies mild}
Let $f_{\mathrm{in}}\in X_{0,1}$ be nonnegative and $\control\in \U$. Then \eqref{eqn: cf_equation}  admits a unique mild solution. Such solution coincides with the solution in the sense of Definition \ref{def: solution}.
\end{proposition}

\begin{proof}
The existence and uniqueness of a local mild solution on a maximal interval $[0,t^\star)$ is standard, since $\F$ generates a $C_0$-semigroup, $\control\in L^\infty (0,T)$ and
the coagulation operator is locally Lipschitz on $X_{0,1}$, see, for example, \cite[Theorem 4.3.3]{cazenave1998semilinear}. Let $\hat\f$ be such a solution. We show that it coincides with the solution in the sense of Definition \ref{def: solution} over $[0,t^\star)$. Since, by Theorem \ref{theorem: well posedness}, $\f$ exists on the whole interval $[0,T]$ and has bounded $X_{0,1}$ norm there, this will imply that $t^\star=T$.

\smallskip
\noindent\textbf{Truncated problems.}
For every $n\in\N$, we consider the truncated kernels
\begin{equation}\label{eqn: truncated kernels}
    \alpha_n(x)\coloneqq \mathds{1}_{[0,n]}(x)\alpha(x),
    \qquad
    K_n(x,y)\coloneqq
    \begin{cases}
        K(x,y), & \text{if }x+y\le n,\\[2pt]
        0, & \text{otherwise}.
    \end{cases}
\end{equation}
We denote by $\F_n$ and $\C_n$ the corresponding fragmentation and coagulation operators. We also introduce the truncated initial condition $f_{\mathrm{in}}^n\coloneqq \mathds{1}_{[0,n]}f_{\mathrm{in}}$. Since $\alpha_n$ is bounded and the expected number of fragments is finite by \eqref{eqn:finite number of fragments}, it is straightforward to check that $\F_n\in \mathcal L(X_{0,1})$.
We denote by $(\T_n(t))_{t\ge 0}$ the $C_0$-semigroup generated by $\F_n$.

Let $\f_n$ be the solution in $[0,T]$ in the sense of Definition \ref{def: solution} of the truncated coagulation-fragmentation equation
\begin{equation}\label{eqn: truncated cf equation}
\begin{cases}
    \partial_t \f_n(t)=\F_n\f_n(t)+\control(t)\,\C_n\f_n(t),\\
    \f_n(0)=f_{\mathrm{in}}^n,
\end{cases}
\end{equation}
which exists and is unique by Remark \ref{remark: existence}. On the other hand the mild solution $\hat \f _n$ only exists on the maximal interval $[0,t_n^\star)$. Since $\F_n$ is bounded and $\C_n$ is continuous on $X_{0,1}$, the mild and integral
formulations are equivalent for \eqref{eqn: truncated cf equation}; see
\cite[Chapter~I, Section~3]{engel2000one}. In particular, for every
$t\in [0,t_n^\star)$,
\begin{equation}\label{eqn: truncated duhamel formulation}
    \hat\f_n(t)
    =
    \T_n(t)f_{\mathrm{in}}^n
    +
    \int_0^t \T_n(t-\tau)\,\control(\tau)\,\C_n\hat\f_n(\tau)\,d\tau.
\end{equation}
Given the boundedness of $\F_n$, this is equivalent to the integral formulation
\begin{equation}\label{eqn: truncated integral formulation}
    \hat\f_n(t)
    =
    f_{\mathrm{in}}^n
    +
    \int_0^t \bigl[\F_n\hat\f_n(\tau)+\control(\tau)\C_n\hat\f_n(\tau)\bigr]\,d\tau,
    \qquad t\in[0,t_n^\star).
\end{equation}
Applying the representation theorem for Bochner integrals in $L$-spaces
(see \cite[Theorem 2.39]{banasiak2006perturbations}), the integral identity
\eqref{eqn: truncated integral formulation} admits an almost-everywhere pointwise
representative, which is precisely the pointwise integral formulation \textbf{(S5)}. Moreover, this formulation implies absolute continuity in time for almost every $x$, and hence \textbf{(S3)} holds. Continuity, together with the fact that the kernels $K_n$ and $\alpha_n$ are bounded and compactly supported, also ensures the integrability conditions \textbf{(S4)}. Nonnegativity of mild solutions of coagulation-fragmentation equations is standard; see, for instance, \cite[Theorem 8.1.1]{banasiak2019analytic}. Therefore, by uniqueness of solutions, $\hat \f_n$ coincides with $\f _n$ on $[0,t_n^\star)$. Furthermore, the same a priori estimates as in Theorem~\ref{theorem: well posedness},
applied to the truncated equation, prevent $\norm{\hat \f_n (t)}_{0,1}$ from blowing up as $t\to t_n^\star$. Therefore, $\hat \f_n$ is globally defined on $[0,T]$ and there coincides with $\f_n$.

\smallskip
\noindent\textbf{Convergence of the truncated semigroups.}
Let $f\in \mathrm{dom}(\F)$. Since $\alpha_n(x)\to \alpha(x)$ pointwise, $\alpha_n(x)\le \alpha(x)$ and the expected number of fragments is finite, a straightforward application of the dominated convergence theorem shows that
\begin{equation*}
    \norm{\F_n f-\F f}_{0,1}\xrightarrow[n\to\infty]{}0.
\end{equation*}
Moreover, the semigroups $(\T_n(t))_{t\geq 0}$ satisfy a uniform exponential bound on $X_{0,1}$. Indeed, let $\mathbf g_n(t)=\T_n(t)g_{in}$ with $ g_{in}\ge 0$. By mass conservation of fragmentation \eqref{eqn:local conservation of mass}, $M_1(\mathbf g_n(t))$ is constant.
On the other hand, since $\alpha_n(y)\leq \alpha_0 y^\lambda$ with $\lambda\in[0,1)$, by \eqref{eqn:finite number of fragments}, we have
\begin{equation*}
        \frac{d}{dt}M_0(\mathbf g_n(t))
    =
    \int_0^\infty \bigl(N_\nu-1\bigr)\alpha_n(y) \mathbf g_n(t,y)\,dy \leq
    (N_\nu-1)\alpha_0 \|\mathbf g_n(t)\|_{0,1}
\end{equation*}
It follows from Gronwall's lemma that
\begin{equation*}
        \|\mathbf g_n(t)\|_{0,1}\le e^{Ct}\|g_{in}\|_{0,1},
    \qquad t\geq 0.
\end{equation*}
By positivity of $\T_n(t)$, the same estimate extends to arbitrary $g_{in}\in X_{0,1}$, and therefore $\T_n(t)$ are uniformly exponentially bounded.
Since also $(\T_\F(t))_{t\geq 0}$ is exponentially bounded on $X_{0,1}$ and $\mathrm{dom}(F)$ is a core for $F$ the Trotter-Kato approximation theorem \cite[Theorem III.4.8]{engel2000one} applies and yields
\begin{equation}\label{eqn: semigroup convergence}
    \sup_{t\in[0,T]}\norm{\T_n(t)g-\T_\F(t)g}_{0,1}\xrightarrow[n\to\infty]{}0
    \qquad \forall g\in X_{0,1}.
\end{equation}

\smallskip
\noindent\textbf{Convergence of the truncated mild solutions.} We claim that $\f_n$ converges to $\hat\f$ in $C([0,T^\star],X_{0,1})$
for every $0<T^\star<t^\star$. By the standard moment estimates for the truncated equation, there exists $R_{T^\star}$, independent of $n$, such that
\begin{equation*}
        \sup_{n\in\N}\sup_{t\in[0,T^\star]}\norm{\f_n(t)}_{0,1}
    +
    \sup_{t\in[0,T^\star]}\norm{\hat \f(t)}_{0,1}
    \leq R_{T^\star}.
\end{equation*}
We know from Proposition \ref{propo: coagulation operator} that the coagulation operators are locally Lipschitz on $X_{0,1}$, with constants independent of $n$. Hence there exists $L_{T^\star}$ such that 
\begin{equation*}
    \norm{\C_n f-\C_n g}_{0,1}\le L_{T^\star} \norm{f-g}_{0,1}
\end{equation*}
for every $n\in\N$ and every $f,g\in X_{0,1}$ with $\norm{f}_{0,1},\norm{g}_{0,1}\le R_{T^\star}$.

Subtracting the mild formulations satisfied by $\f_n$ and $\hat \f$, taking norms and using the uniform semigroup bound and the local Lipschitz estimate for $\C_n$ , we obtain
\begin{align*}
    \norm{\f_n(t)-\hat \f(t)}_{0,1}
    \le
    \varepsilon_n(T^\star)
    +
    e^{CT^\star}u_{\max}L_{T^\star}
    \int_0^t \norm{\f_n(\tau)-\hat \f(\tau)}_{0,1}\,d\tau,
\end{align*}
where
\begin{align*}
    \varepsilon_n(T^\star)
    \coloneqq
    &\sup_{t\in[0,T^\star]}\norm{\T_n(t)f_{\mathrm{in}}^n-\T_\F(t)f_{\mathrm{in}}}_{0,1}\\
    &+
    \sup_{t\in[0,T^\star]}
    \int_0^t
    \norm{\bigl(\T_n(t-\tau)-\T_\F(t-\tau)\bigr)\control(\tau)\C\hat \f(\tau)}_{0,1}\,d\tau\\
    &+
    \sup_{t\in[0,T^\star]}
    \int_0^t
    \norm{\T_n(t-\tau)\control(\tau)\bigl(\C_n\hat \f(\tau)-\C\hat \f(\tau)\bigr)}_{0,1}\,d\tau.
\end{align*}
It is easy to see that $\varepsilon_n(T^\star)\to 0$. Indeed, for the first term, it is enough to use the uniform semigroup bound, \eqref{eqn: semigroup convergence}, and the fact that $f_{\mathrm{in}}^n\to f_{\mathrm{in}}$ in $X_{0,1}$. The second term tends to zero by \eqref{eqn: semigroup convergence} and dominated convergence. Finally, since $K_n(x,y)\to K(x,y)$ pointwise and $|K_n(x,y)|\leq |K(x,y)|$, for every fixed $\tau\in[0,T^\star]$ one has
\begin{equation*}
        \norm{\C_n\hat \f(\tau)-\C\hat \f(\tau)}_{0,1}\xrightarrow[n\to\infty]{}0.
\end{equation*}
Moreover, since $K$ is bounded, easy computations show that
\begin{equation*}
    \|\C_n\hat \f(\tau)-\C\hat \f(\tau)\|_{0,1}
\le
\|\C_n\hat \f(\tau)\|_{0,1}+\|\C\hat \f(\tau)\|_{0,1}
\le
4\|K\|_\infty \|\hat \f(\tau)\|_{0,1}^2.
\end{equation*}
Since $\hat \f$ is bounded in $C([0,T^\star],X_{0,1})$, the integrand is dominated by an $L^1(0,T^\star)$ function independent of $n$.
Hence, dominated convergence yields the desired convergence. Gronwall's lemma now yields
\begin{equation*}
        \sup_{t\in[0,T^\star]}\norm{\f_n(t)-\hat \f(t)}_{0,1}\xrightarrow[n\to\infty]{}0.
\end{equation*}

\smallskip
\noindent\textbf{Conclusion.}
By Remark \ref{remark: existence} and Theorem \ref{thm: weak dependence on controls}, the truncated solutions $\f_n$ converge to $\f$ in $C\bigl([0,T^\star],X_{0,w}\bigr)$. On the other hand, the previous step shows that
\begin{equation*}
        \f_n \xrightarrow[n\to\infty]{} \hat \f
    \qquad\text{in }C\bigl([0,T^\star],X_{0,1}\bigr),
\end{equation*}
hence also in $C\bigl([0,T^\star],X_{0,w}\bigr)$. Therefore, given the arbitrariness of $T^\star$ we deduce that $\f=\hat \f$ on $[0,t^\star)$, as desired.
\end{proof}

A first consequence of the mild representation is the following Lipschitz dependence result.

\begin{proposition}[Lipschitz dependence on the control]\label{propo: Lipschitz dependence}
Let $\control_1,\control_2\in\U$. Then there exists a constant $\mathbf{L}_{cf}>0$, depending only on
$T$, $\alpha$, $b$, $K$, $u_{\max}$, such that
\begin{equation*}
    \sup_{t\in[0,T]}\norm{\f_{\control_2}(t)-\f_{\control_1}(t)}_{0,1}
    \;\leq\; \mathbf{L}_{cf}\,\norm{\control_2-\control_1}_{2}.
\end{equation*}
\end{proposition}

\begin{proof}
We notice that since $K$ is bounded, the bound in Theorem \ref{theorem: well posedness} ensures that the maps
$t\mapsto \C \f_{\control_i}$ from $[0,T]$ to $X_{0,1}$ are well defined and integrable. Therefore,  given that $\F$ generates a $C_0$ semigroup in $X_{0,1}$ and by the Duhamel formula, both solutions satisfy the mild formulation
\begin{equation*}
    \f_{\control_i}(t)
    =\mathcal T_{\F}(t)f_{\mathrm{in}}
    +\int_0^t\mathcal T_{\F}(t-\tau)\,\control_i(\tau)\,\C \f_{\control_i}(\tau)\,d\tau,
\qquad i=1,2.
\end{equation*}
Set $\mathbf{w}:=\f_{\control_2}-\f_{\control_1}$. Subtracting the two identities gives
\begin{equation*}
    \mathbf{w}(t)=\int_0^t\mathcal T_{\F}(t-\tau)
    \Big((\control_2-\control_1)(\tau)\,\C \f_{\control_2}(\tau)
     +\control_1(\tau)\big(\C \f_{\control_2}(\tau)-\C \f_{\control_1}(\tau)\big)\Big)\,d\tau.
\end{equation*}
Theorem \ref{theorem: well posedness} yields a uniform bound for $\norm{\C \f_{\control_i}(\tau)}_{0,1}$. Using the semigroup bound, and the local Lipschitz estimate for $\C$ from Proposition~\ref{propo: coagulation operator}, we then obtain a constant $C>0$ for which
\begin{equation*}
    \norm{\mathbf{w}(t)}_{0,1}
    \;\leq\; C\int_0^t\lvert\control_2-\control_1\rvert(\tau)\,d\tau
      + C\int_0^t\norm{\mathbf{w}(\tau)}_{0,1}\,d\tau.
\end{equation*}
Grönwall's lemma yields
\begin{equation*}
    \sup_{t\in[0,T]}\norm{\mathbf{w}(t)}_{0,1}
    \;\leq\; C_T\int_0^T \lvert\control_2-\control_1\rvert(\tau)\,d\tau,
\end{equation*}
for a constant $C_T>0$ depending only on $T$, $u_{\max}$ and the kernels. Finally, Cauchy–Schwarz gives the desired inequality.
\end{proof}

We now turn our attention to the linearised dynamics. Since $\control(\cdot)D\mathcal C[\f_\control(\cdot)]$ can be seen as a bounded operator-valued perturbation of the fragmentation operator $\F$, standard semigroup arguments (e.g., the Dyson--Phillips expansion, see \cite{pazy2012semigroups}) yield the following well-posedness result for the linearised equation.

\begin{proposition}[Linearised equation]\label{propo: linearized system}
For $\control\in\U$, consider the linear Cauchy problem
\begin{equation}\label{eqn: linearized system}
\begin{cases}
\partial_t v(t)=\mathcal F\, v(t) + \control(t)\,D\mathcal C[\f_\control(t)]\,v(t) + g(t),\\[2pt]
v(s)=\hat v,
\end{cases}
\end{equation}
with $s\in[0,T]$, $\hat v\in X_{0,1}$, and $g\in L^1(0,T;X_{0,1})$.
Then \eqref{eqn: linearized system} is well-posed. More precisely, there exists
a unique strongly continuous evolution family
\begin{equation*}
    \Phi_\control \colon \{(t,s): 0\leq s\leq t\leq T\}\to \mathcal{L}(X_{0,1})
\end{equation*}
characterised by the integral equation
\begin{equation}\label{eqn:Phi-forward}
\Phi_\control(t,s)
=
\mathcal T_{\F}(t-s)
+
\int_s^t \mathcal T_{\F}(t-\tau)\,
\control(\tau)\,D\C[\f_\control(\tau)]\,\Phi_\control(\tau,s)\,\mathrm{d}\tau.
\end{equation}
Moreover, for every $\hat v\in X_{0,1}$, the map
$t\mapsto \Phi_\control(t,s)\hat v$ is the unique mild solution of the
homogeneous problem \eqref{eqn: linearized system} corresponding to $g\equiv 0$.
For general $g\in L^1(0,T;X_{0,1})$, the unique mild solution is given by the
variation-of-constants formula
\begin{equation*}
    v(t)=\Phi_\control(t,s)\hat v + \int_s^t \Phi_\control(t,\tau)\,g(\tau)\,\mathrm{d}\tau.
\end{equation*}

Moreover, there exist constants
$\hat M>0$, $\hat\omega\in\R$, independent of $\control\in\U$, such that
\begin{equation}\label{eqn: growth bound evolution family}
\|\Phi_\control(t,s)\|\leq \hat M e^{\hat\omega(t-s)}
\qquad\forall\,0\leq s\le t\leq T.
\end{equation}
\end{proposition}

Having established well-posedness of the linearised equation, standard semilinear arguments show that it indeed describes the first-order variation of the controlled dynamics around a reference trajectory.

\begin{proposition}\label{propo: linearisation}
Let
\begin{equation*}
    \control^\varepsilon = \control^\star + \varepsilon \delta \control,
    \qquad
    f_{\mathrm{in}}^\varepsilon = f_{\mathrm{in}} + \varepsilon \hat v_0,
\end{equation*}
and assume that $\control^\varepsilon\in \U$ and $f_{\mathrm{in}}^\varepsilon\geq 0$ for all $\varepsilon$ sufficiently small. Let further $\f^\varepsilon$ be the corresponding solution of \eqref{eqn: cf_equation}. Then
\begin{equation*}
    \f^\varepsilon = \f_{\control^\star}+\varepsilon \delta \f + o(\varepsilon)
    \qquad \text{in } C([0,T],X_{0,1}),
\end{equation*}
where $\delta \f$ is the unique mild solution of
\begin{equation}\label{eqn :lin-inhom}
\partial_t v(t)=\F v(t) + \control^\star(t)\,D\C[\f_{\control^\star}(t)]\,v(t)
 + \delta \control(t)\,\C \f_{\control^\star}(t),
 \qquad
 \delta \f(0)=\hat v_0.
\end{equation}
\end{proposition}
\begin{proof}
Set $\f^\star\coloneqq \f_{\control^\star}$ and define $\mathbf g^\varepsilon \coloneqq \f^\star+\varepsilon \delta\f$ and $    \rho^\varepsilon \coloneqq \f^\varepsilon-\mathbf g^\varepsilon$. By the mild formulations for $\f^\varepsilon$, $\f^\star$, and $\delta\f$, one obtains for every $t\in[0,T]$
\begin{equation*}
    \rho^\varepsilon(t)
    =
    \int_0^t \T_\F(t-s)
    \Bigl[
        \control^\varepsilon(s)\bigl(\C\f^\varepsilon(s)-\C\mathbf g^\varepsilon(s)\bigr)
        +\Gamma^\varepsilon(s)
    \Bigr]\,ds,
\end{equation*}
where
\begin{align*}
    \Gamma^\varepsilon(s)
    \coloneqq{}&
    \control^\star(s)\Bigl(
        \C\mathbf g^\varepsilon(s)-\C\f^\star(s)
        -\varepsilon D\C[\f^\star(s)]\,\delta\f(s)
    \Bigr)
    +\varepsilon \delta\control(s)
    \Bigl(
        \C\mathbf g^\varepsilon(s)-\C\f^\star(s)
    \Bigr).
\end{align*}

Since $\f^\star,\delta\f\in C([0,T],X_{0,1})$, the family $\bigl\{\mathbf g^\varepsilon(t):\, t\in[0,T],\ |\varepsilon|\le \varepsilon_0\bigr\}$
is bounded in $X_{0,1}$ for $\varepsilon_0>0$ small enough. By the continuous Fréchet differentiability of $\C$ on $X_{0,1}$, there exists a remainder $r^\varepsilon\in C([0,T],X_{0,1})$ such that
\begin{equation*}
    \C\mathbf g^\varepsilon
    =
    \C\f^\star
    +\varepsilon D\C[\f^\star]\,\delta\f
    +r^\varepsilon,
    \qquad
    \|r^\varepsilon\|_{C([0,T],X_{0,1})}=o(\varepsilon).
\end{equation*}
Moreover, since $\C$ is locally Lipschitz on bounded sets, there exists $L>0$ such that
\begin{equation*}
    \|\C\f^\varepsilon(t)-\C\mathbf g^\varepsilon(t)\|_{X_{0,1}}
    \le
    L\,\|\rho^\varepsilon(t)\|_{X_{0,1}}
    \qquad \forall t\in[0,T].
\end{equation*}
Using also that
\begin{equation*}
    \|\C\mathbf g^\varepsilon-\C\f^\star\|_{C([0,T],X_{0,1})}=O(\varepsilon),
\end{equation*}
we infer from the definition of $\Gamma^\varepsilon$ that
\begin{equation*}
    \|\Gamma^\varepsilon\|_{C([0,T],X_{0,1})}=o(\varepsilon).
\end{equation*}
Therefore, recalling the exponential bound on $\T_\F$ and the uniform boundedness of the controls $\control^\varepsilon$, there exists $C>0$ such that
\begin{equation*}
    \|\rho^\varepsilon(t)\|_{X_{0,1}}
    \le
    C\int_0^t \|\rho^\varepsilon(s)\|_{X_{0,1}}\,ds
    +
    C\int_0^t \|\Gamma^\varepsilon(s)\|_{X_{0,1}}\,ds.
\end{equation*}
By Gr\"onwall's lemma,
\begin{equation*}
    \|\rho^\varepsilon\|_{C([0,T],X_{0,1})}
    \le
    C_T\,\|\Gamma^\varepsilon\|_{C([0,T],X_{0,1})}
    =
    o(\varepsilon),
\end{equation*}
which ends the proof.
\end{proof}

\subsection{The adjoint equation}

We now turn to the backward adjoint equation. In light of
Remark~\ref{remark: adjoint fragmentation} and the nonreflexivity of
$X_{0,1}$, the adjoint fragmentation semigroup $\mathcal{T}_{\mathcal
F}^{*}$ is only weak-$*$ continuous; accordingly, the adjoint Cauchy
problem is formulated in the weak-$*$ sense.

\begin{proposition}\label{propo: adjoint Cauchy problem}
For every $\hat{\varphi}\in X_{\infty,1}$ and $\control\in\U$, the backward adjoint
problem
\begin{equation}\label{eqn: adjoint Cauchy}
\begin{cases}
\partial_t \varphi(t)= -\mathcal F^{*}\varphi(t)-\control(t)\,D\mathcal C[\f_\control(t)]^{*}\varphi(t),\\[2pt]
\varphi(T)=\hat \varphi,
\end{cases}
\end{equation}
is well-posed in $X_{\infty,1}$. There exists a unique mild weak-$*$ continuous
$\bfphi$ such that, for all
$f\in X_{0,1}$ and $s\in[0,T]$,
\begin{equation}\label{eqn: mild weak * solution}
  \langle \bfphi(s), f\rangle
  = \langle \mathcal T^*_{\mathcal F}(T-s)\hat{\varphi}, f\rangle
  +\int_{s}^{T} \big\langle \mathcal T^*_{\F}(\tau-s)\control(\tau)\,
   D\mathcal C[\f_\control(\tau)]^{*}\,\bfphi(\tau),\, f\big\rangle\,\mathrm{d}\tau .
\end{equation}
Moreover, $\bfphi(s)=\Phi^*_\control(s,T)\hat{\varphi}$, where
$\Phi^*_\control(s,t)\coloneqq \Phi_\control(t,s)^{*}$ is the adjoint evolution family,
and there exist $\hat M>0$, $\hat\omega\in\R$, independent of $\control\in\U$, such
that
\begin{equation}\label{eq:adjoint-bound}
  \norm{\Phi^*_\control(s,t)}\;\leq\; \hat M\,\mathrm e^{\hat\omega\,(t-s)}
  \qquad (0\leq s\leq t\le T).
\end{equation}
Solutions depend Lipschitz continuously on the control and on the terminal condition. More precisely, for every $R>0$ there exists a constant $\mathbf{L}_{\Phi,R}>0$ such that
\begin{equation}\label{eq:adjoint-lip}
  \sup_{t\in[0,T]}\|\bfphi_2(t)-\bfphi_1(t)\|_{\infty,1}
  \leq \mathbf{L}_{\Phi,R}\Bigl(\norm{\control_2 - \control_1}_2 + \norm{\hat \varphi_2-\hat \varphi_1}_{\infty,1}\Bigr)
\end{equation}
for every $\control_1,\control_2\in\U$ and every $\hat \varphi_1,\hat \varphi_2 \in X_{\infty,1}$ such that $\|\hat\varphi_1\|_{\infty,1},\|\hat\varphi_2\|_{\infty,1}\leq R$, where for $i=1,2$, $\bfphi_i$ denotes the solution of \eqref{eqn: adjoint Cauchy} associated with control $\control_i$ and terminal condition $\hat \varphi_i$.
\end{proposition}

\begin{proof}
For brevity, set
\begin{equation*}
    B_\control(t)\coloneqq \control(t)\,D\mathcal C[\f_\control(t)]
\in \mathcal L(X_{0,1}), \qquad t\in[0,T].
\end{equation*}
Since by Theorem~\ref{theorem: well posedness} $\f_\control$ is uniformly bounded in $X_{0,1}$, and
$f\mapsto D\mathcal C[f]$ is linear and bounded from $X_{0,1}$ to
$\mathcal L(X_{0,1})$, there exists a constant $C_B>0$, independent of
$\control\in\U$, such that
\begin{equation}\label{eq:uniform-B-bound}
\|B_\control(t)\|_{\mathcal L(X_{0,1})}\le C_B
\qquad\text{for a.e. }t\in[0,T].
\end{equation}

\smallskip

\textbf{Evolution family and bound.}
Since $\Phi_\control$ is a strongly continuous evolution family on $X_{0,1}$,
its adjoints $\Phi_\control^*(s,t)\coloneqq \Phi_\control(t,s)^*$ form an evolution family on $X_{\infty,1}=X_{0,1}^*$.
For every $\hat\varphi\in X_{\infty,1}$ and $f\in X_{0,1}$, the map
\begin{equation*}
    s\longmapsto \langle \Phi_\control^*(s,T)\hat\varphi,f\rangle
=
\langle \hat\varphi,\Phi_\control(T,s)f\rangle
\end{equation*}
is continuous because $s\mapsto \Phi_\control(T,s)f$ is continuous in $X_{0,1}$.
Hence $\Phi_\control^*(\cdot,T)\hat\varphi$ is weak-$*$ continuous.
Also,
\begin{equation*}
    \|\Phi_\control^*(s,t)\|
=
\|\Phi_\control(t,s)^*\|
=
\|\Phi_\control(t,s)\|
\le \hat M\,e^{\hat\omega (t-s)},
\end{equation*}
which proves \eqref{eq:adjoint-bound}.

\smallskip

\textbf{Existence and mild formula.}
Define $\bfphi(s)\coloneqq \Phi_\control^*(s,T)\hat\varphi$ for $s\in[0,T]$. Then $\bfphi$ is weak-$*$ continuous by the previous step.
To derive the mild formula, pairing the integral equation \eqref{eqn:Phi-forward} against $f\in X_{0,1}$ gives
\begin{align*}
\langle \bfphi(s),f\rangle
=
\langle \hat\varphi,\Phi_\control(T,s)f\rangle
=
\langle \hat\varphi,\mathcal T_{\mathcal F}(T-s)f\rangle
+\int_s^T
\big\langle \hat\varphi,\Phi_\control(T,\tau)\,B_\control(\tau)\,
\mathcal T_{\mathcal F}(\tau-s)f\big\rangle\,\mathrm d\tau.
\end{align*}
The integrand is absolutely integrable in $\tau$: indeed, by
\eqref{eq:adjoint-bound}, \eqref{eq:uniform-B-bound}, and the semigroup bound for
$\mathcal T_{\mathcal F}$,
\begin{equation*}
    \big|
\langle \hat\varphi,\Phi_\control(T,\tau)\,B_\control(\tau)\,
\mathcal T_{\mathcal F}(\tau-s)f\rangle
\big|
\le
\|\hat\varphi\|_{\infty,1}\,\hat M e^{\hat\omega (T-\tau)}\,
C_B\,M_{\mathcal F}e^{\omega_{\mathcal F}(\tau-s)}\,\|f\|_{0,1},
\end{equation*}
and the right-hand side belongs to $L^1(s,T)$.
Using the adjoint relation twice, we obtain
\begin{align*}
\big\langle \hat\varphi,\Phi_\control(T,\tau)\,B_\control(\tau)\,
\mathcal T_{\mathcal F}(\tau-s)f\big\rangle=
\big\langle \Phi_\control^*(\tau,T)\hat\varphi,\,
B_\control(\tau)\,\mathcal T_{\mathcal F}(\tau-s)f\big\rangle=
\big\langle \mathcal T_{\mathcal F}^*(\tau-s)\,B_\control(\tau)^*\,\bfphi(\tau),\,f\big\rangle.
\end{align*}
Therefore
\begin{equation*}
      \langle \bfphi(s), f\rangle
  = \langle \mathcal T^*_{\mathcal F}(T-s)\hat{\varphi}, f\rangle
  +\int_{s}^{T} \big\langle \mathcal T^*_{\mathcal F}(\tau-s)\control(\tau)\,
   D\mathcal C[\f_\control(\tau)]^{*}\,\bfphi(\tau),\, f\big\rangle\,\mathrm{d}\tau .
\end{equation*}

\smallskip

\textbf{Lipschitz dependence on the control and the terminal condition.}
Fix $R>0$, let $\control_1,\control_2\in\U$, and let
$\hat\varphi_1,\hat\varphi_2\in X_{\infty,1}$ satisfy
$\|\hat\varphi_1\|_{\infty,1},\|\hat\varphi_2\|_{\infty,1}\le R$.
For $i=1,2$, let
\begin{equation*}
    B_i(t)\coloneqq \control_i(t)\,D\mathcal C[\f_{\control_i}(t)],
\qquad
\bfphi_i(t)\coloneqq \Phi_{\control_i}^*(t,T)\hat\varphi_i.
\end{equation*}
Subtracting the two mild identities yields, for every $s\in[0,T]$,
\begin{align*}
\bfphi_2(s)-\bfphi_1(s)
=&
\mathcal T_{\mathcal F}^*(T-s)(\hat\varphi_2-\hat\varphi_1) \\
+&\int_s^T \mathcal T_{\mathcal F}^*(\tau-s)\,(B_2(\tau)-B_1(\tau))^*\,\bfphi_1(\tau)\,\mathrm d\tau 
+\int_s^T \mathcal T_{\mathcal F}^*(\tau-s)\,B_2(\tau)^*
\bigl(\bfphi_2(\tau)-\bfphi_1(\tau)\bigr)\,\mathrm d\tau .
\end{align*}
Taking the norm in $X_{\infty,1}$ and using
$\|\mathcal T_{\mathcal F}^*(t)\|\le M_{\mathcal F}e^{\omega_{\mathcal F}t}$, we get
\begin{align}
\|\bfphi_2(s)-\bfphi_1(s)\|_{\infty,1}
&\leq C
\|\hat\varphi_2-\hat\varphi_1\|_{\infty,1}
\notag\\
&
+ C\int_s^T \|B_2(\tau)-B_1(\tau)\|_{\mathcal L(X_{0,1})}\,
\|\bfphi_1(\tau)\|_{\infty,1}\,\mathrm d\tau
+ C\int_s^T \|\bfphi_2(\tau)-\bfphi_1(\tau)\|_{\infty,1}\,\mathrm d\tau
\label{eq:adjoint-pregronwall}
\end{align}
for some constant $C>0$ independent of $\control_1,\control_2$. Next, since $f\mapsto D\mathcal C[f]$ is linear and bounded,
\begin{align*}
\|B_2(\tau)-B_1(\tau)\|_{\mathcal L(X_{0,1})}
&\leq
C\Bigl(
|\control_2(\tau)-\control_1(\tau)|
+\|\f_{\control_2}(\tau)-\f_{\control_1}(\tau)\|_{0,1}
\Bigr)
\end{align*}
for another constant $C>0$.
Moreover, by \eqref{eq:adjoint-bound}, $\|\bfphi_1(\tau)\|_{\infty,1}
\leq \hat M e^{\hat\omega T}\|\hat\varphi_1\|_{\infty,1}
\leq \hat M e^{\hat\omega T}R$. Combining the previous two estimates, Hölder's inequality, and
Proposition~\ref{propo: Lipschitz dependence}, we infer that
\begin{align*}
\int_s^T \|B_2(\tau)-B_1(\tau)\|_{\mathcal L(X_{0,1})}\,\mathrm d\tau
&\le
C\int_0^T |\control_2(\tau)-\control_1(\tau)|\,\mathrm d\tau
+ C\int_0^T \|\f_{\control_2}(\tau)-\f_{\control_1}(\tau)\|_{0,1}\,\mathrm d\tau \\
&\le
C\,T^{1/2}\|\control_2-\control_1\|_2
+ C\,T\,\sup_{t\in[0,T]}
\|\f_{\control_2}(t)-\f_{\control_1}(t)\|_{0,1} \\
&\le
C_R\,\|\control_2-\control_1\|_2 .
\end{align*}
Inserting this into \eqref{eq:adjoint-pregronwall}, we obtain
\begin{equation*}
    \|\bfphi_2(s)-\bfphi_1(s)\|_{\infty,1}
\le
C_R\|\hat\varphi_2-\hat\varphi_1\|_{\infty,1}
+
C_R\|\control_2-\control_1\|_2
+
C_R\int_s^T \|\bfphi_2(\tau)-\bfphi_1(\tau)\|_{\infty,1}\,\mathrm d\tau.
\end{equation*}
A backward Grönwall argument yields
\begin{equation*}
    \sup_{t\in[0,T]}\|\bfphi_2(t)-\bfphi_1(t)\|_{\infty,1}
\le
\mathbf L_{\Phi,R}
\Bigl(
\|\control_2-\control_1\|_2
+
\|\hat\varphi_2-\hat\varphi_1\|_{\infty,1}
\Bigr),
\end{equation*}
for some constant $\mathbf L_{\Phi,R}>0$ depending only on $R$ and on the data.

Finally, uniqueness follows by applying the previous estimate with
$\control_1=\control_2$ and $\hat\varphi_1=\hat\varphi_2$.
\end{proof}

\subsection{Proof of Theorem~\ref{thrm: gradients}}

We now prove the gradient representation \eqref{eq:grad-explicit} and the
Lipschitz continuity of $\control \mapsto\nabla_u J(\control)$.

\begin{proof}[Proof of Theorem~\ref{thrm: gradients}]
The proof is divided into several steps.

\textbf{Gradient formula.}
Denote by $\f_{\control^\star}$ and $\f^\varepsilon$ the corresponding solutions of
\eqref{eqn: cf_equation}. By Proposition~\ref{propo: linearisation}, we can linearise the dynamics along $f_{\control^\star}$ and obtain
\begin{equation*}
    \f^\varepsilon = \f_{\control^\star}+\varepsilon \delta \f + o(\varepsilon)
\quad\text{in }C([0,T],X_{0,1}),
\end{equation*}
where
\begin{equation}\label{eqn: delta f VoC}
    \delta f(t)=\int_0^t \Phi_{\control^\star}(t,\tau)\, \delta \control(\tau)\,\mathcal{C}\f_{\control^\star}(\tau)\,\mathrm{d}\tau.
\end{equation}
By Fréchet differentiability of $\psi$ in $X_{0,1}$,
\begin{equation}\label{eqn: epsilon expansion of psi}
    \psi(\f^\varepsilon(T))=\psi(\f_{\control^\star}(T))
    +\varepsilon\,\big\langle D\psi(\f_{\control^\star}(T)),\,\delta \f(T)\big\rangle
    +o(\varepsilon).
\end{equation}
Using \eqref{eqn: delta f VoC}, the growth bound
\eqref{eqn: growth bound evolution family} and the boundedness of
$\mathcal C \f_{\control^\star}$, the integrand is Bochner integrable in $X_{0,1}$ and
Fubini–-Tonelli theorem yields
\begin{equation}\label{eqn: scalar product expansion}
    \begin{split}
    \big\langle D\psi(\f_{\control^\star}(T)),\,\delta \f(T)\big\rangle
    &=\int_0^T \delta \control(\tau)\,\Big\langle D\psi(\f_{\control^\star}(T)),\,
        \Phi_{\control^\star}(T,\tau)\,\mathcal C \f_{\control^\star}(\tau)\Big\rangle \mathrm{d}\tau\\
    &=\int_0^T \delta \control(\tau)\,\Big\langle \Phi_{\control^\star}^*(\tau,T)D\psi(\f_{\control^\star}(T)),\,
        \mathcal C \f_{\control^\star}(\tau)\Big\rangle \mathrm{d}\tau,
    \end{split}
\end{equation}
where we used duality in the second identity. Setting
\begin{equation*}
    \bfphi_{\control^\star}(\tau)\coloneqq\Phi_{\control^\star}^*(\tau,T)\,D\psi(\f_{\control^\star}(T)),
\end{equation*}
Proposition~\ref{propo: adjoint Cauchy problem} shows that $\bfphi_{\control^\star}$
is the unique weak-$*$ mild solution of \eqref{eqn: adjoint Cauchy
problem}. Inserting \eqref{eqn: scalar product expansion} into
\eqref{eqn: epsilon expansion of psi} and adding the contribution of the
running cost gives
\begin{equation*}
    \frac{\mathrm{d}}{\mathrm{d}\varepsilon}\Big|_{\varepsilon=0}J(\control^\varepsilon)
=\big\langle t\mapsto w(\control^\star(t)-1)+\langle\bfphi_{\control^\star}(t),\mathcal C \f_{\control^\star}(t)\rangle,\ \delta \control\big\rangle_2,
\end{equation*}
which yields the expression of $\nabla_u J(\control^\star)$ and
\eqref{eqn: directional derivative J}.

\smallskip
\textbf{Lipschitz continuity.}
Let, for $i=1,2$, $\bfphi_i$ be the solution of \eqref{eqn: adjoint Cauchy problem}
associated with the control $\control_i$ and terminal condition $\hat\varphi_i=D\psi\bigl(\f_{\control_i}(T)\bigr)$. Assume that $D\psi$ is Lipschitz on bounded subsets of $X_{0,1}$.
Since $\{\f_\control(T):\control\in\U\}$ is bounded in $X_{0,1}$, the set $\{D\psi(\f_\control(T)):\control\in\U\}$
is bounded in $X_{\infty,1}$, say by some $R>0$. Hence, by
Proposition~\ref{propo: adjoint Cauchy problem}, the family
$\{\bfphi_\control\}_{\control\in\U}$ is uniformly bounded in
$L^\infty(0,T;X_{\infty,1})$ and
\begin{equation*}
    \sup_{t\in[0,T]}\|\bfphi_{\control_2}(t)-\bfphi_{\control_1}(t)\|_{\infty,1}
\leq \mathbf L_{\Phi,R}\Bigl(
\|\control_2-\control_1\|_2
+\|D\psi(\f_{\control_2}(T))-D\psi(\f_{\control_1}(T))\|_{\infty,1}
\Bigr).
\end{equation*}
Since $D\psi$ is Lipschitz on the bounded set $\{\f_\control(T):\control\in\U\}$, and
Proposition~\ref{propo: Lipschitz dependence} gives
\begin{equation*}
    \|\f_{\control_2}(T)-\f_{\control_1}(T)\|_{0,1}\leq \mathbf{L}_{cf}\|\control_2-\control_1\|_2,
\end{equation*}
there exists $L>0$ such that
\begin{equation*}
    \sup_{t\in[0,T]}\|\bfphi_{\control_2}(t)-\bfphi_{\control_1}(t)\|_{\infty,1}
\le L\|\control_2-\control_1\|_2.
\end{equation*}

For $t\in[0,T]$ we estimate
\begin{equation*}
    \begin{aligned}
\big|\langle \bfphi_{\control_2}(t),\mathcal C \f_{\control_2}(t)\rangle
      -\langle \bfphi_{\control_1}(t),\mathcal C \f_{\control_1}(t)\rangle\big|
&\le \|\mathcal C \f_{\control_2}(t)\|_{0,1}\,
      \|\bfphi_{\control_2}(t)-\bfphi_{\control_1}(t)\|_{\infty,1}\\
&\quad + \|\bfphi_{\control_1}(t)\|_{\infty,1}\,
      \|\mathcal C \f_{\control_2}(t)-\mathcal C \f_{\control_1}(t)\|_{0,1}.
\end{aligned}
\end{equation*}
By \eqref{eqn: M_0 estimate} and Proposition~\ref{propo: coagulation operator},
there exists $C>0$ such that
\begin{equation*}
    \sup_{t\in[0,T]}\|\mathcal C \f_{\control_i}(t)\|_{0,1}\le C,
\qquad
\sup_{t\in[0,T]}\|\bfphi_{\control_i}(t)\|_{\infty,1}\le C,
\end{equation*}
and
\begin{equation*}
    \sup_{t\in[0,T]}
\|\mathcal C \f_{\control_2}(t)-\mathcal C \f_{\control_1}(t)\|_{0,1}
\le C\sup_{t\in[0,T]}\|\f_{\control_2}(t)-\f_{\control_1}(t)\|_{0,1}.
\end{equation*}
Using the previous estimate for $\bfphi_{\control_2}-\bfphi_{\control_1}$ and
Proposition~\ref{propo: Lipschitz dependence}, we deduce that there exists
$L>0$ such that
\begin{equation*}
    \sup_{t\in[0,T]}
\big|\langle \bfphi_{\control_2}(t),\mathcal C \f_{\control_2}(t)\rangle
     -\langle \bfphi_{\control_1}(t),\mathcal C \f_{\control_1}(t)\rangle\big|
\le L\|\control_2-\control_1\|_2.
\end{equation*}
Finally,
\begin{equation*}
    \|\nabla_u J(\control_2)-\nabla_u J(\control_1)\|_2
\le w\|\control_2-\control_1\|_2
+\big\|\langle \bfphi_{\control_2},\mathcal C \f_{\control_2}\rangle
      -\langle \bfphi_{\control_1},\mathcal C \f_{\control_1}\rangle\big\|_2
\le \mathbf L\|\control_2-\control_1\|_2,
\end{equation*}
for some $\mathbf L>0$, which proves the claimed Lipschitz continuity of
$\control\mapsto \nabla_u J(\control)$.
\end{proof}

\subsection{Proof of Theorem \ref{thm: PMP}}

We end this section with the proof of the Pontryagin minimum principle for \eqref{eqn: OCP}.

\begin{proof}
Let $\control^\star$ be locally optimal in $\U$. If $\control\in \U$, since $\U$ is convex we have that for every $\varepsilon\in[0,1]$
\begin{equation*}
    \control^\varepsilon \coloneqq \control^\star + \varepsilon (\control-\control^\star)\in \U.
\end{equation*}
By local optimality, $J\bigl(\control^\varepsilon\bigr)-J(\control^\star)\geq 0$
for all $\varepsilon$ sufficiently small. Theorem \ref{thrm: gradients} along the direction $\delta \control = \control-\control^\star$ yields
\begin{equation*}
    J\bigl(\control^\varepsilon\bigr)
    = J\bigl(\control^\star\bigr)
      +\varepsilon\bigl\langle\nabla_u J(\control^\star), \delta \control\bigr\rangle_2
      + o (\varepsilon),
\end{equation*}
dividing by $\varepsilon$ and taking the limit for $\varepsilon\to 0$ finally gives the variational inequality
\begin{equation}\label{eqn: VI}
    \langle \nabla_u J(\control^\star), \control -\control ^\star\rangle_2\geq 0
\end{equation}

Fix a Borel set $A\subset [0,T]$, $\omega\in [u_{\min}, u_{\max}]$. We introduce
the control
\begin{equation}\label{lemma: test variation}
    \control_{A, \omega}(t)\coloneqq 
    \begin{cases}
        \omega, & \text{ if }t\in  A,\\[2pt]
        \control^\star(t) &\text{otherwise}.
    \end{cases}
\end{equation}
By specialising \eqref{eqn: VI} at $\control=\control_{A, \omega}$ and recalling the expression of the gradient of $J$ at $\control^\star$ we obtain
\begin{equation*}
    \int_A \Bigl(w\bigl(\control^\star(t)-1)+\langle \bfphi^\star(t), \C\f^\star(t)\rangle\Bigr)(\omega-\control^\star(t))dt\geq 0
\end{equation*}
Therefore, for every fixed $\omega\in [u_{\min},u_{\max}]$,
\begin{equation*}
    \Bigl(
w(\control^\star(t)-1)
+\langle \bfphi^\star(t),\C\f^\star(t)\rangle
\Bigr)
(\omega-\control^\star(t))
\ge 0
\qquad\text{for a.e. }t\in[0,T].
\end{equation*}
By a standard density argument using $\mathbb{Q}\cap [u_{\min},u_{\max}]$ and the
continuity of $H$ in $\omega$, one may choose the exceptional null set
independently of $\omega$. Equivalently, for a.e.\ $t\in[0,T]$ and every $\omega\in [u_{\min},u_{\max}]$,
\begin{equation}\label{eq:pointwise-fo}
    \partial_u H\bigl(\f^\star(t),\bfphi^\star(t),\control^\star(t)\bigr)
    \,(\omega-\control^\star(t))\ge 0.
\end{equation}
Now, for a.e. $t$, the map
\begin{equation*}
    \omega \longmapsto H\bigl(\f^\star(t),\bfphi^\star(t),\omega\bigr)
\end{equation*}
is convex on $[u_{\min},u_{\max}]$, and
\eqref{eq:pointwise-fo} is precisely its first-order optimality condition at
$\omega=\control^\star(t)$. Hence
\begin{equation*}
    H\bigl(\f^\star(t), \bfphi^\star(t), \control^\star(t)\bigr)
\leq
H\bigl(\f^\star(t), \bfphi^\star(t), \omega\bigr)
\qquad
\forall\,\omega\in [u_{\min},u_{\max}]
\end{equation*}
for almost every $t\in[0,T]$, which is the desired minimum condition.

When $w>0$, the Hamiltonian is strictly convex in $\omega$ and the
unconstrained minimiser is
\begin{equation*}
    \control_0(t)
    = 1 - \tfrac{1}{w}\,\bigl\langle \bfphi^\star(t), \mathcal{C}\f^\star(t)\bigr\rangle.
\end{equation*}
Projecting this minimiser onto the admissible control set
$[u_{\min},u_{\max}]$ yields the explicit formula \eqref{eqn: feedback}
for almost every $t\in[0,T]$.
\end{proof}

\begin{remark}
A singular–arc analysis in the unregularised case $w=0$ is beyond the scope
of the present work.
\end{remark}

\section{Numerical experiments}\label{sec:numerics}

This section illustrates the optimal control problem~\eqref{eqn: OCP} on a single test case and confirms the qualitative behaviour predicted by our analysis. The focus is on a proof-of-concept rather than on numerical analysis or high–order discretisations. All simulations are implemented in \texttt{Python} and accelerated with JAX’s just-in-time compilation. The implementation is available at \url{https://github.com/enricosartor/coagulation-fragmentation-control}.

\subsection{Optimisation scheme}\label{subsec:scheme}

We consider the minimisation problem
\begin{equation}\label{eqn:minprob}
\min_{\control\in\U} J(\control)=\min_{\control\in\U}
\frac{w}{2}\int_{0}^{T}\bigl(\control(t)-1\bigr)^{2}\,dt+
\int _{x_{\min}}^{x_{\max}}\f_\control(T,x) dx
\end{equation}
subject to the coagulation–fragmentation dynamics~\eqref{eqn: cf_equation}. The goal is thus to minimise the number of particles whose sizes lie between $x_{\min}$ and $x_{\max}$. We recall that the admissible controls form the box 
\begin{equation*}
    \U=\bigl\{\control \in L^2(0,T) \colon \control(t)\in [u_{\min},u_{\max}] \text{ for a.e. } t\in(0,T)\bigr\}.
\end{equation*}
which is closed and convex in $L^2(0,T)$.

Given the gradient formula in Theorem~\ref{thrm: gradients}, we solve
\eqref{eqn:minprob} by projected gradient descent. This leads to the classical
forward–backward sweep:
\begin{enumerate}[label=(\roman*)]
\item integrate the state equation forward in time for a given control;
\item integrate the adjoint equation backward in time with terminal condition $\varphi(T)=D\psi(f(T))$;
\item update the control by a projected gradient step.
\end{enumerate}
We adopt an \emph{optimise–then–discretise} strategy: the continuous state and
adjoint are both evaluated on the same finite volume grid and time step (see
§\ref{subsec:discretisation}), and all pairings use the corresponding
quadrature. In particular, the implemented search direction is the
discretisation of the continuous gradient; we do not derive a separate discrete
adjoint.

For completeness, we record the adjoint equation in its \emph{formal strong
form}:
\begin{equation}\label{eqn: adjoint_equation}
    \begin{split}
        \partial_t\varphi(t,x) =&- \int_0^\infty \control(t) K(x, y) \f_\control(t, y)\, \varphi(t, x+y)\, dy\\
        &+\int_0^\infty \control(t) K(y,x) \f_\control(t, y)\, \varphi(t,y)\,dy
          + \varphi(t, x)\int _0^\infty \control (t) K(x,y) f_\control(t, y)\, dy\\
        &+\alpha(x)\varphi(t,x)-\int_0^x \alpha(x) b(y,x)\, \varphi(t, y)\, dy.
    \end{split}
\end{equation}
In the analysis, the adjoint variable is constructed only as a weak-$*$ solution as in Proposition~\ref{propo: adjoint Cauchy problem}; no pointwise meaning of \eqref{eqn: adjoint_equation} is assumed. The expression above should therefore
be understood as the formal representation of the adjoint operator, which is the one we discretise in the numerical scheme.

\subsubsection*{Projected gradient algorithm}

The projection $P_\U$ onto $\U$ acts pointwise as
\begin{equation*}
      P_\U(\control)(t)=P_{[u_{\min}, u_{\max}]}\bigl(\control(t)\bigr).
\end{equation*}
Stationary points of \eqref{eqn:minprob} are characterised by the fixed–point relation
\begin{equation*}
      \control^\star=P_{\U}\bigl(\control^\star-\tau\,\nabla _uJ(\control^\star)\bigr)
  \qquad\text{for any }\tau>0,
\end{equation*}
which motivates the \emph{gradient mapping}
\begin{equation*}
      G_{\eta}(\control)\coloneqq\frac{1}{\eta}\Bigl(\control-P_{\U}\bigl(\control-\eta\nabla_u J(\control)\bigr)\Bigr)
\end{equation*}
and the stopping criterion $\|G_{1}(u_k)\|_{2}\le\varepsilon$. The Lipschitz continuity of $\nabla _u J$ on bounded subsets of $\U$ (Theorem~\ref{thrm: gradients}) implies the standard descent estimate. In particular, projected gradient descent with Armijo backtracking produces a monotonically decreasing sequence $J(u_k)$ and satisfies $\|G_{\eta_k}(u_k)\|_2\to0$; see, e.g., \cite[Theorem 2.4]{hinze2008optimization}. Algorithm~\ref{alg: pgd} summarises one optimisation step.

\begin{remark}
The convergence result established above concerns the optimisation procedure defined at the continuous level, that is, using the gradient of the continuous reduced cost functional. By contrast, for the numerical implementation based on an \emph{optimise-then-discretise} strategy, no convergence result is proved here. In particular, we do not claim convergence of the resulting discrete optimisation algorithm to a local minimiser of the discretised problem.
\end{remark}

\begin{algorithm}[ht]
  \caption{Projected gradient descent}
  \label{alg: pgd}
  \begin{algorithmic}[1]
    \STATE Choose an initial admissible control $u_0$, parameters $\eta_0>0$, $\beta\in(0,1)$, $\sigma\in(0,1)$, and a tolerance $\varepsilon>0$.
    \FOR{$k=0,1,2,\dots$}
      \STATE \textbf{Forward pass:} Integrate \eqref{eqn: cf_equation} with control $u_k$ to obtain the state $f_k$.
      \STATE \textbf{Backward pass:} Set $\varphi_k(T)=D\psi\bigl(f_k(T)\bigr)$ and integrate \eqref{eqn: adjoint_equation} backward to obtain $\varphi_k$.
      \STATE \textbf{Gradient:} Compute, for each $t$,
      \begin{equation*}
        g_k(t)=w\bigl(u_k(t)-1\bigr)+\langle \varphi_k(t),\mathcal{C}f_k(t)\rangle_{\Delta x}.
      \end{equation*}
      \STATE \textbf{Stopping test:} If
      \begin{equation*}
        \bigl\|u_k-P_{\U}(u_k-g_k)\bigr\|_2=\|G_1(u_k)\|_2 \leq \varepsilon,
      \end{equation*}
      then terminate.
      \STATE \textbf{Line search:} Choose the largest $\eta_k=\eta_0\beta^n$ such that
      \begin{equation*}
        J\bigl(P_{\U}(u_k-\eta_k g_k)\bigr)
        \leq
        J(u_k)-\sigma\,\eta_k\,\|G_{\eta_k}(u_k)\|_2^2,
      \end{equation*}
      where
      \begin{equation*}
          G_{\eta_k}(u_k)=\frac{1}{\eta_k}\Bigl(u_k-P_{\U}(u_k-\eta_k g_k)\Bigr).
      \end{equation*}
      \STATE \textbf{Update:}
      \begin{equation*}
          u_{k+1}=P_{\U}(u_k-\eta_k g_k).
      \end{equation*}
    \ENDFOR
  \end{algorithmic}
\end{algorithm}

\subsection{Numerical discretisation}\label{subsec:discretisation}

Both state and adjoint equations are discretised on the same uniform finite–volume grid on $[0,25]$ with $N_x=800$ cells, spacing $\Delta x=25/N_x$, and cell centres $x_i=(i-\tfrac12)\Delta x$. Spatial inner products reuse the finite–volume quadrature,
\begin{equation*}
    \langle \varphi, f\rangle_{\Delta x}=\sum_{i=1}^{N_x} \varphi_i\,f_i\,\Delta x,
\end{equation*}
so that the discrete pairing mirrors the $X_{0,1}$–$X_{\infty,1}$ duality. Time stepping uses forward Euler with $\Delta t=0.005$ for both the forward and backward equations. We denote by $J_{\Delta t, \Delta x}$ the discrete counterpart of the cost $J$.

\begin{remark}In our numerical experiments we employ a straightforward finite–volume discretisation of the coagulation–fragmentation equation. This scheme does not enforce exact conservation of the total mass at the discrete level; we verified, however, that the resulting mass defect remains small on the grids used; in the test case of Section 5.3, the relative mass loss does not exceed $10^{-3}$. We emphasise that designing and analysing structure–preserving numerical methods is not the focus of the present work. Conservative schemes that preserve the behaviour of selected moments of the solution (for instance, the total number of particles and their total mass) are available in the literature, see, for example, the fixed–pivot technique in \cite{kumar1996solution}.
\end{remark}

\begin{remark}
The adjoint system in Theorem~\ref{thrm: gradients} is formulated in $X_{\infty,1}$, a weighted-$L^\infty$-type dual space. Because $X_{\infty,1}$ is the dual of the \emph{nonreflexive} weighted $L^1$ space $X_{0,1}$, the adjoint variables exist only in a mild weak-$*$ sense (see Proposition~\ref{propo: adjoint Cauchy problem}). Together with the nonseparability of $X_{\infty,1}$ and the fact that the domain of the adjoint fragmentation operator is only weak-$\ast$ dense, this raises delicate numerical issues. A full finite volume error analysis of the backward problem is therefore beyond the scope of this paper; our computations are intended merely to confirm the qualitative behaviour predicted by the theory.
\end{remark}

We validate the gradients on the fully discretised objective $J_{\Delta t,\Delta x}$ using a Taylor test.
Given a unit direction $d$ (normalised with respect to the discrete $L^2$ inner product in time), we consider
\begin{equation}\label{eq:taylor_residual}
E_{\Delta t,\Delta x}(\varepsilon)\coloneqq
\frac{\bigl|J_{\Delta t,\Delta x}(u+\varepsilon d)-J_{\Delta t,\Delta x}(u)
-\varepsilon\langle g(u),d\rangle_{\Delta t}\bigr|}{\varepsilon},
\qquad
\langle a,b\rangle_{\Delta t}\coloneqq \sum_{k} a_k b_k \Delta t.
\end{equation}
When $g=g^{\mathrm{disc}}_{\Delta t,\Delta x}(u)\coloneqq \nabla_u J_{\Delta t,\Delta x}(u)$ is computed by automatic
differentiation in JAX, the residual satisfies $E_{\Delta t,\Delta x}(\varepsilon)=\mathcal O(\varepsilon)$.
In contrast, when $g=g^{\mathrm{cont}}_{\Delta t,\Delta x}(u)$ is obtained by discretising the continuous adjoint system,
$E_{\Delta t,\Delta x}(\varepsilon)$ exhibits a plateau as $\varepsilon\downarrow 0$, reflecting the usual
optimise--then--discretise vs.\ discretise--then--optimise mismatch at fixed resolution.

To quantify this discrepancy we also compare the two gradients directly. Since the continuous adjoint yields a time-density,
we use the time-discretised quantity $\Delta t\, g^{\mathrm{cont}}_{\Delta t,\Delta x}(u)$ and report the relative error
\begin{equation}\label{eq:rho_grad}
\rho_{\Delta t,\Delta x}\coloneqq
\frac{\bigl\|\Delta t\, g^{\mathrm{cont}}_{\Delta t,\Delta x}(u)-g^{\mathrm{disc}}_{\Delta t,\Delta x}(u)\bigr\|_2}
{\bigl\|g^{\mathrm{disc}}_{\Delta t,\Delta x}(u)\bigr\|_2}.
\end{equation}
The plateau level of \eqref{eq:taylor_residual} and the mismatch \eqref{eq:rho_grad} decrease under time refinement on the
fixed spatial grid of the test case; see Table~\ref{tab:grad_validation}.

\begin{table}[t]
\centering
\caption{Gradient validation on the spatial grid used in the main test case.
We report the plateau level of the Taylor residual $E_{\Delta t, \Delta x}(\varepsilon)$ and the relative mismatch $\rho_{\Delta t, \Delta x}$ between 
the discrete gradient $\nabla_u J_{\Delta t, \Delta x}$ (JAX autodiff) and the time-discretised adjoint gradient.}
\label{tab:grad_validation}
\begin{tabular}{lcc}
\toprule
Time step & plateau of $E_{\Delta t, \Delta x}(\varepsilon)$ & $\rho_{\Delta t, \Delta x}$ \\
\midrule
$2\Delta t$   & $1.1\times 10^{-2}$ & $9.2\times 10^{-2}$ \\
$\Delta t$ & $6.0\times 10^{-3}$ & $6.5\times 10^{-2}$ \\
$\Delta t/2$ & $3.0\times 10^{-3}$ & $4.6\times 10^{-2}$ \\
\bottomrule
\end{tabular}
\end{table}

\subsection{Test case: Avoiding the formation of fine particles.}

We now specialise the cost functional~\eqref{eqn:minprob} to a standard
objective in aerosol engineering: reducing the concentration of fine and
ultrafine particles. In many applications (combustion, spray drying, atmospheric
chemistry), particles simultaneously \emph{coagulate} (grow by collisions) and
\emph{fragment} (break under shear or chemical reaction). Because sub–PM$_{2.5}$
material is tightly regulated for health and climate reasons, we choose a
terminal cost that penalises the number of particles of size $x\leq
x_{\max}\coloneqq 5.0$ at the final time $T=1$, and seek controls that make this
quantity as small as possible.

The kernels and initial data are
\begin{equation}\label{eq:kernels_C}
  K(x,y)=\frac{1}{20}(1+x)^{\frac{1}{4}}(1+y)^{\frac{1}{4}}\mathds{1}_{[0,25]}(x+y), 
  \qquad
  \alpha(x)=\frac{1}{5}\sqrt{x}, 
  \qquad
  b(x,y)=\frac{2}{y},
\end{equation}
\begin{equation}\label{eq:init_C}
  f_{\mathrm{in}}(x)=2\,\mathds{1}_{[0,25]}(x)\,
         \exp\bigl[-(x-7.5)^2/50\bigr].
\end{equation}
Coagulation is moderately size dependent and cut off beyond $x+y=25$ to avoid out-of-grid 
loss of mass, while fragmentation grows with size and produces on
average $N_0=2$ fragments per break–up (binary fragmentation).

\vspace{0.1cm}

\textbf{Algorithmic performance.}
We apply Algorithm~\ref{alg: pgd} with $u_0\equiv 1$ and tolerance
$\varepsilon=0.075$, using Armijo backtracking on the gradient mapping. We set $\eta_0=0.1$, $\beta=\tfrac12$, and $\sigma=10^{-4}$. For
the reference choice $w=1.0$ the method converges in $15$ iterations.

We monitor the loss $J(u_k)$ (Fig.~\ref{fig:loss}) together with three
residuals (Fig.~\ref{fig:residuals}): the projection residual
\begin{equation*}
    \|u_k - P_{\mathcal U}(u_k - g_k)\|_2,
\end{equation*}
and two Pontryagin-like residuals, namely the relative control error
\begin{equation*}
    r_k \coloneqq \frac{\norm{u_k - u_k^\star}_2}{\norm{u_k^\star}_2},
    \qquad
    u_k^\star(t)\coloneqq P_{[u_{\min},u_{\max}]}\left(
    1-\frac{1}{w}\,\bigl\langle\varphi_k(t), \mathcal C f_k(t)\bigr\rangle_{\Delta x}
    \right),
\end{equation*}
and the relative Hamiltonian error
\begin{equation*}
    s_k \coloneqq \frac{\norm{H(f_k,\varphi_k,u_k)-H_k^\star}_2}
    {\norm{H_k^\star}_2}, 
    \qquad 
    H_k^\star(t) \coloneqq \min_{\omega\in[u_{\min}, u_{\max}]} 
    H\bigl(f_k(t), \varphi_k(t), \omega\bigr).
\end{equation*}
As expected from the discussion in Section~\ref{subsec:scheme}, the projection
residual decreases monotonically. Moreover, the PMP residuals also attain very
small values, providing a posteriori evidence that the algorithm has converged
to a local minimiser.

\begin{figure}[ht]
  \centering
  \begin{subfigure}[t]{0.48\textwidth}
    \centering
    \includegraphics[width=\linewidth]{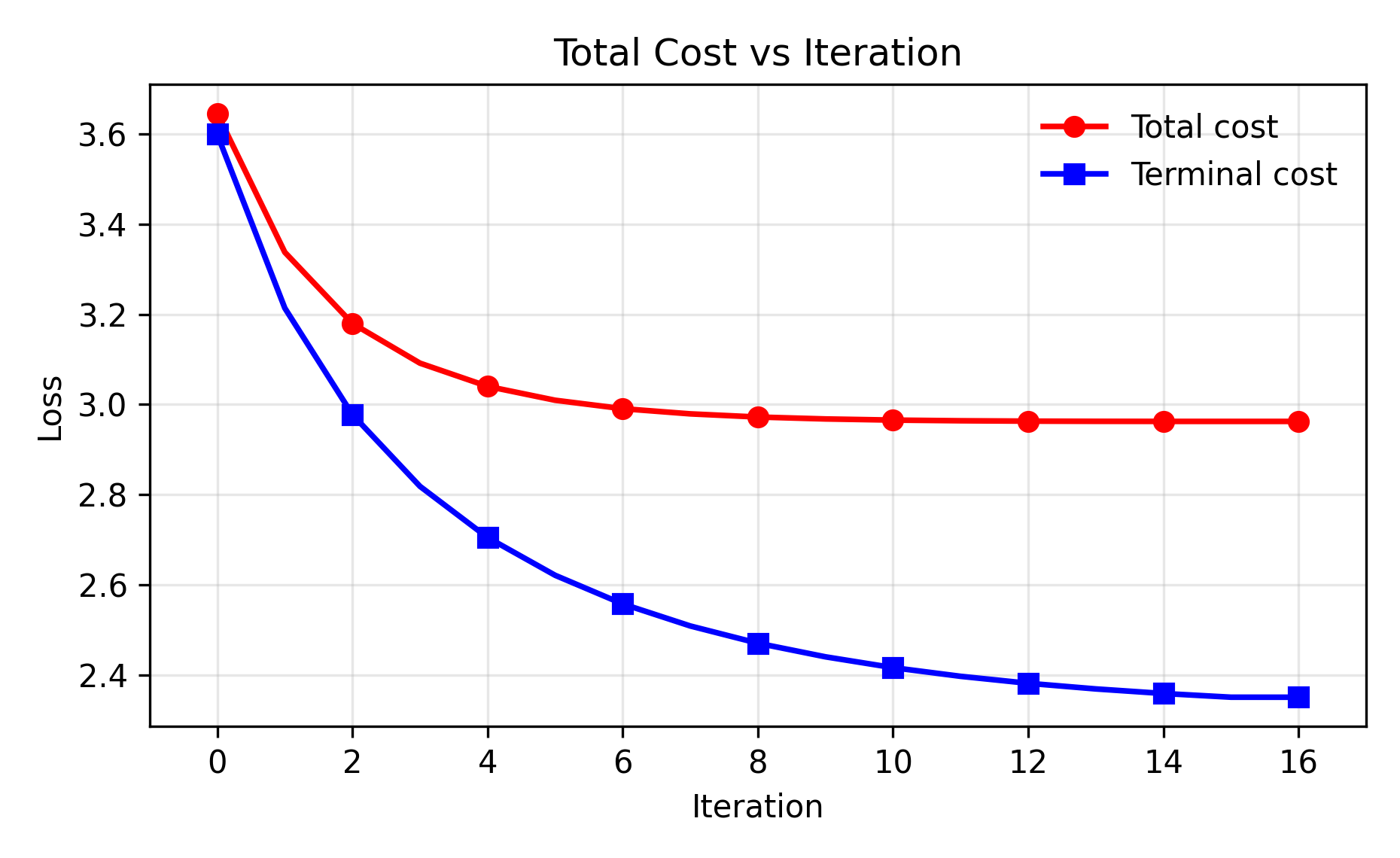}
    \caption{Loss}
    \label{fig:loss}
  \end{subfigure}\hfill
  \begin{subfigure}[t]{0.48\textwidth}
    \centering
    \includegraphics[width=\linewidth]{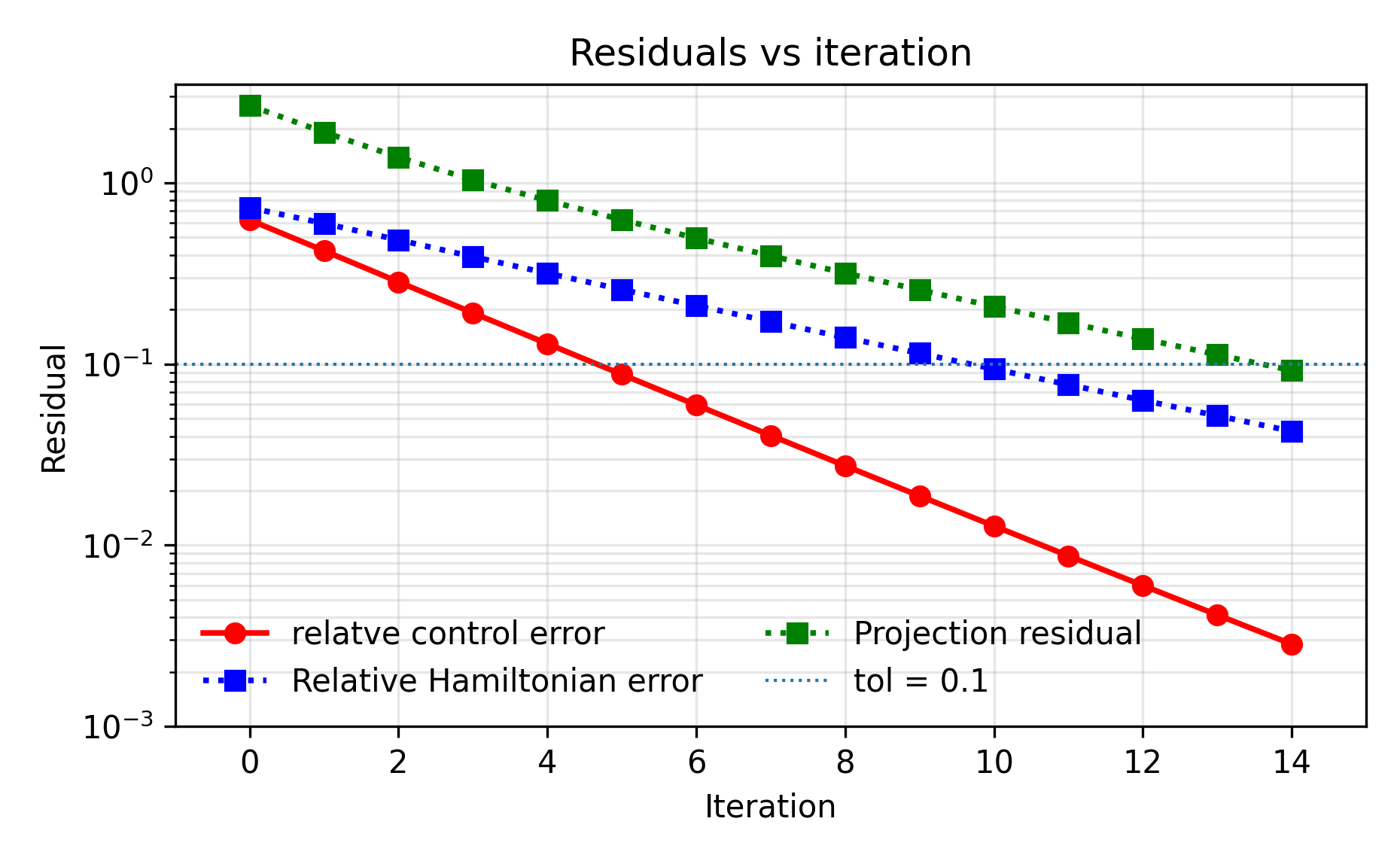}
    \caption{Residuals}
    \label{fig:residuals}
  \end{subfigure}
  \caption{Objective value and residuals along the iterations of
    Algorithm~\ref{alg: pgd}.}
  \label{fig:loss-residuals}
\end{figure}

\vspace{0.1cm}

\textbf{Effect of the regularisation weight.}
To illustrate the trade-off between control effort and terminal performance, we
repeat the experiment for several values of the weight $w$
in~\eqref{eqn:minprob}. Table~\ref{tab:weight-sweep} reports, for each $w$, the
optimal value $J(u^\star)$, the terminal cost (number of particles within $[0.0, 5.0]$), and
the number of iterations of Algorithm~\ref{alg: pgd}. As expected, the higher
the weight, the less aggressive the control is.

\begin{table}[ht]
  \centering
  \begin{tabular}{cccc}
    \toprule
    $w$ 
      & total cost
      & terminal cost
      & iterations \\
    \midrule
    $0.2$ & 2.003 & 1.572 & 34 \\
    $1.0$ & 2.9609 & 2.390 & 15 \\
    $5.0$ & 3.799 & 3.406 & 4 \\
    \midrule
    no control & 4.357 & 4.357 \\
    \bottomrule
  \end{tabular}
  \caption{Effect of the regularisation weight $w$ on the optimal cost and particles in the target region.}
  \label{tab:weight-sweep}
\end{table}

\vspace{0.1cm}

\textbf{Optimal control and terminal densities.}
Figure~\ref{fig:controls-densities} summarises the outcome for the different
values of the weight $w$. The locally optimal controls produced by the
projected-gradient algorithm (Figure~\ref{fig:controls_A}) significantly modify
the dynamics and consistently outperform the uncontrolled case. The right panel
(Figure~\ref{fig:density_A}) compares the initial distribution (black) with the
uncontrolled and controlled terminal distributions. The optimal controls
effectively reduce the number of particles inside the target region
$[0.0,5.0]$, in line with the modelling goal. In all cases, the control activity
focuses near the final time $T=1$, enhancing coagulation when it is most
effective for depleting the target window.

\begin{figure}[ht]
  \centering
  \begin{subfigure}[t]{0.48\textwidth}
    \centering
    \includegraphics[width=\linewidth]{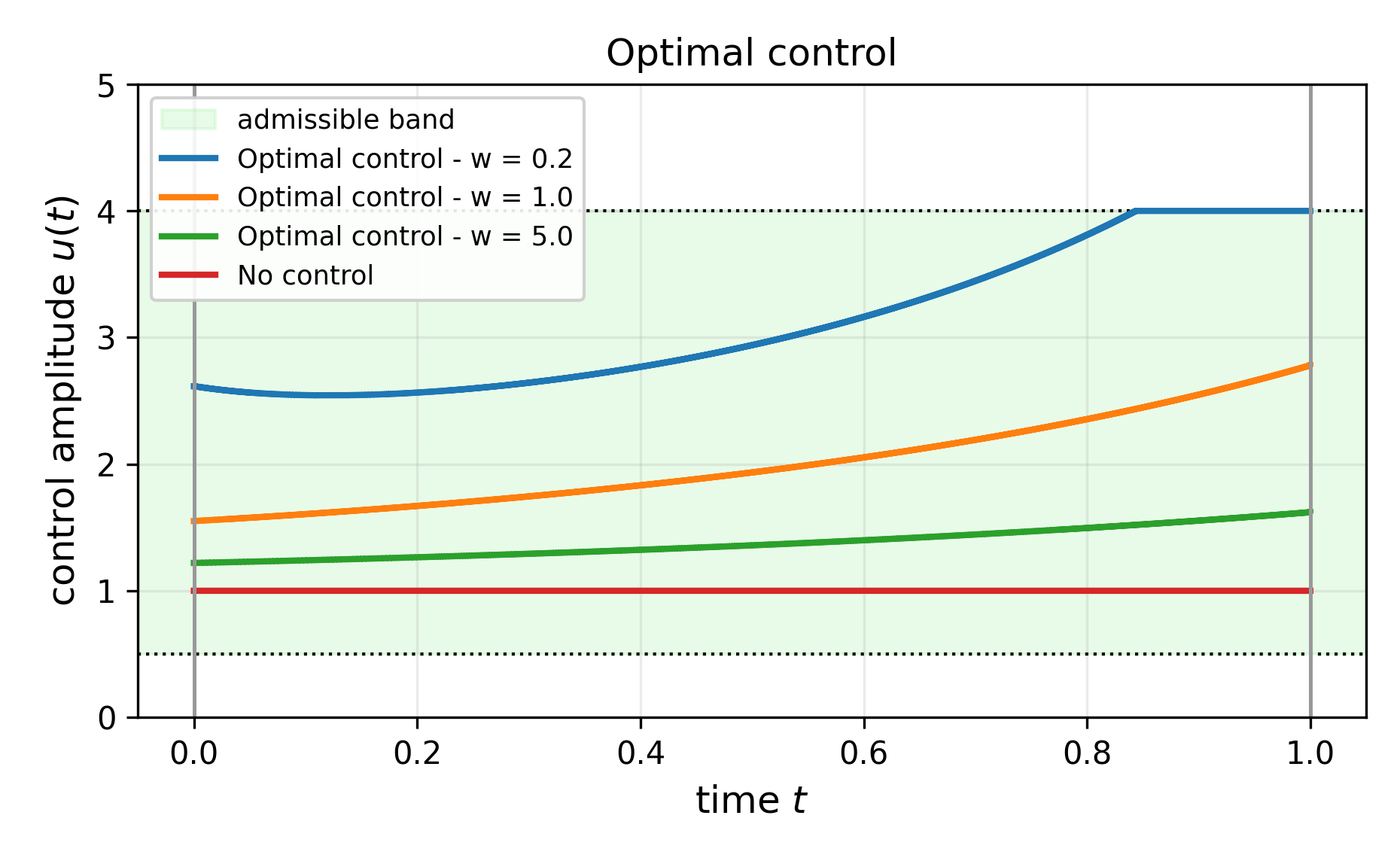}
    \caption{(Locally) optimal controls with different weights.}
    \label{fig:controls_A}
  \end{subfigure}\hfill
  \begin{subfigure}[t]{0.48\textwidth}
    \centering
    \includegraphics[width=\linewidth]{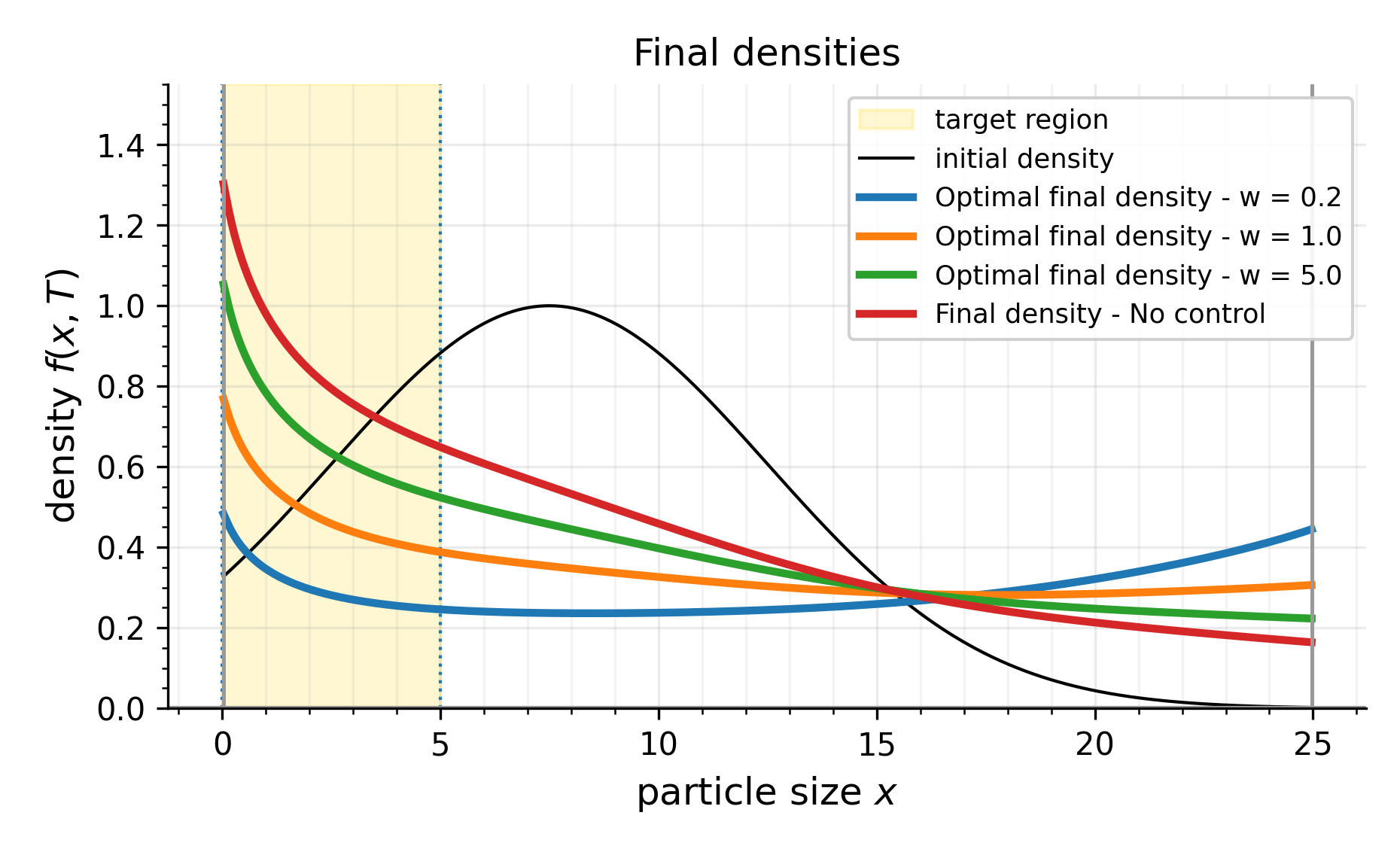}
    \caption{Terminal densities with controls and without.}
    \label{fig:density_A}
  \end{subfigure}
  \caption{ (Locally) optimal controls and corresponding terminal densities.}
  \label{fig:controls-densities}
\end{figure}

\section{Conclusion and perspectives}

We have studied a coagulation–fragmentation model in which a scalar control
$\control(t)\in[u_{\min},u_{\max}]$ rescales the coagulation kernel,
$\hat K(t,x,y)=\control(t)\,K(x,y)$. Working in the weighted space $X_{0,1}$, we established well-posedness of the controlled dynamics and
existence of optimal controls. A $\Gamma$-convergence result then yields
robustness of optimal solutions with respect to perturbations of the kernels,
with particular emphasis on approximating unbounded coagulation kernels by
truncations.

In the case of bounded coagulation kernels, linearising along a trajectory and
formulating the adjoint in the dual space $X_{\infty,1}$ leads to a
weak-$\ast$ well-posed backward equation and a Pontryagin-type minimum
principle, together with a gradient representation.  Under a natural Lipschitz
assumption on the terminal cost, this gradient depends Lipschitz continuously
on the control, and at the continuous level this regularity yields convergence
of the projected-gradient method with Armijo backtracking. Section~\ref{sec:numerics}
provides a proof-of-concept finite-volume implementation of this strategy. The
numerical experiment shows that a single low-dimensional actuator can
substantially reshape the size distribution within a prescribed window.

Possible directions for future work include:
\begin{enumerate}[label=(\alph*)]
  \item \textbf{More general kernels.} Extend the analysis to coagulation and
  fragmentation kernels $K$, $\alpha$ and $b$ that allow for \emph{gelation} or
  \emph{shattering}. A natural goal in this
  regime is to quantify and ultimately minimise the mass loss caused by these
  extreme phenomena.

  \item \textbf{Control of fragmentation.}  
  Allow the optimiser to modulate the break-up rate $\alpha$ in addition to
  scaling the coagulation kernel. This second actuator can restore influence
  when $\control$ alone is weak and permits steering of the size spectrum by
  promoting or suppressing the production of small clusters. The resulting
  bilinear–affine control problem, with two interacting adjoints, raises both
  analytic and numerical questions, while closely reflecting industrial
  practice.

  \item \textbf{Size-dependent controls.} A natural way to enrich the control action is to let it depend on particle size. A fully general actuator
  $u(t,x,y)$ acting on all pairs of sizes, however, is neither very practical
  nor easy to identify in applications. A more structured option is to let the
  control depend only on the total size of the colliding pair, for example
  through a law of the form
  $u(t,x,y)=\hat u\bigl(t,a(x+y)\bigr)$
  with a given (possibly nonlinear) function $a$. In this setting the optimiser can favour or penalise the creation of particles according to
  their total size, while retaining a relatively low-dimensional control
  description. Such size-dependent actuators would likely enlarge the reachable set and raise interesting questions
  concerning controllability, optimisation, and numerical approximation.
  \item \textbf{Finite-volume analysis.}
  On the numerical side, it would be of interest to develop a finite-volume
  analysis for the adjoint equation, together with a complete error analysis
  for the resulting gradient-based optimisation scheme.
\end{enumerate}

\section*{Acknowledgements}

Part of this work was carried out while the author was affiliated with the Chair for Dynamics, Control, Machine Learning and Numerics at Friedrich-Alexander University Erlangen-Nürnberg. The author gratefully acknowledges the Chair for its support and stimulating research environment. He is especially grateful to Professor Enrique Zuazua for introducing him to the topic and for his valuable guidance and suggestions.

\section*{Data Availability Statement}

No experimental data were collected for this study. The code used to generate the numerical experiments and supporting simulation data is openly available on GitHub at \url{https://github.com/enricosartor/coagulation-fragmentation-control}. An archived version of the repository is available at DOI: \url{https://doi.org/10.5281/zenodo.19593092}.

\bibliographystyle{plain}
\bibliography{bibliography}

@book{banasiak2019analytic,
  title={Analytic {M}ethods for {C}oagulation-{F}ragmentation {M}odels, {V}olume {II}},
  author={Banasiak, J. and Lamb, W. and Lauren{\c{c}}ot, P.},
  year={2019},
  publisher={Chapman and Hall/CRC}
}

@book{banasiak2019analytic1,
  title={{A}nalytic {M}ethods for {C}oagulation-{F}ragmentation {M}odels, {V}olume {I}},
  author={Banasiak, J. and Lamb, W. and Lauren{\c{c}}ot, P.},
  year={2019},
  publisher={Chapman and Hall/CRC}
}

@book{flory1953principles,
  title={Principles of polymer chemistry},
  author={Flory, P. J.},
  year={1953},
  publisher={Cornell University Press}
}

@book{matyjaszewski2002handbook,
  title={Handbook of radical polymerization},
  author={Matyjaszewski, K. and Davis, T. P.},
  volume={922},
  year={2002},
  publisher={Wiley Online Library}
}

@article{jacobson2005enhanced,
  title={Enhanced coagulation due to evaporation and its effect on nanoparticle evolution},
  author={Jacobson, M. Z. and Kittelson, D. B. and Watts, W. F.},
  journal={Environmental Science \& Technology},
  volume={39},
  number={24},
  pages={9486--9492},
  year={2005},
  publisher={ACS Publications}
}

@book{hinds2022aerosol,
  title={Aerosol technology: properties, behavior, and measurement of airborne particles},
  author={Hinds, W. C. and Zhu, Y.},
  year={2022},
  publisher={John Wiley \& Sons}
}

@article{smoluchowski1916drei,
  title={Drei {V}ortrage uber {D}iffusion, {B}rownsche {B}ewegung und {K}oagulation von {K}olloidteilchen},
  author={Smoluchowski, M. V.},
  journal={Zeitschrift fur Physik},
  volume={17},
  pages={557--585},
  year={1916}
}

@article{smoluchowski1918versuch,
  title={Versuch einer mathematischen {T}heorie der {K}oagulationskinetik kolloider {L}{\"o}sungen},
  author={Smoluchowski, M. V.},
  journal={Zeitschrift f{\"u}r physikalische Chemie},
  volume={92},
  number={1},
  pages={129--168},
  year={1918},
  publisher={De Gruyter Oldenbourg}
}

@article{muller1928allgemeinen,
  title={Zur allgemeinen {T}heorie der raschen {K}oagulation: {D}ie {K}oagulation von {S}t{\"a}bchen-und {B}l{\"a}ttchenkolloiden; die {T}heorie beliebig polydisperser {S}ysteme und der {S}tr{\"o}mungskoagulation},
  author={M{\"u}ller, H.},
  journal={Kolloidchemische Beihefte},
  volume={27},
  pages={223--250},
  year={1928},
  publisher={Springer}
}

@article{melzak1957scalarI,
  title={A scalar transport equation},
  author={Melzak, Z. A.},
  journal={Transactions of the American Mathematical Society},
  volume={85},
  number={2},
  pages={547--560},
  year={1957},
  publisher={JSTOR}
}

@article{melzak1957scalarII,
  title={A scalar transport equation. {I}{I}.},
  author={Melzak, Z. A.},
  journal={Michigan Mathematical Journal},
  volume={4},
  number={3},
  pages={193--206},
  year={1957},
  publisher={University of Michigan, Department of Mathematics}
}

@article{van1984size,
  title={Size distribution in the polymerisation model {A}f{RB}g},
  author={Van Dongen, P. G. J. and Ernst, M. H.},
  journal={Journal of Physics A: Mathematical and General},
  volume={17},
  number={11},
  pages={2281},
  year={1984},
  publisher={IOP Publishing}
}

@article{aizenman1979convergence,
  title={Convergence to equilibrium in a system of reacting polymers},
  author={Aizenman, M. and Bak, T. A.},
  journal={Communications in Mathematical Physics},
  volume={65},
  number={3},
  pages={203--230},
  year={1979},
  publisher={Springer}
}

@article{ziff1985kinetics,
  title={The kinetics of cluster fragmentation and depolymerisation},
  author={Ziff, R. M. and McGrady, E. D.},
  journal={Journal of Physics A: Mathematical and General},
  volume={18},
  number={15},
  pages={3027},
  year={1985},
  publisher={IOP Publishing}
}

@article{samsel1982kinetics,
  title={Kinetics of rouleau formation. {I}. {A} mass action approach with geometric features},
  author={Samsel, R. W. and Perelson, A. S.},
  journal={Biophysical Journal},
  volume={37},
  number={2},
  pages={493--514},
  year={1982},
  publisher={Elsevier}
}

@article{samsel1984kinetics,
  title={Kinetics of rouleau formation. {II}. {R}eversible reactions},
  author={Samsel, R. W. and Perelson, A. S.},
  journal={Biophysical Journal},
  volume={45},
  number={4},
  pages={805--824},
  year={1984},
  publisher={Elsevier}
}

@book{friedlander2000smoke,
  title={Smoke, {D}ust, and {H}aze: {F}undamentals of {A}erosol {D}ynamics},
  author={Friedlander, S. K.},
  volume={198},
  year={2000},
  publisher={Oxford University Press, New York}
}

@article{vigil1989stability,
  title={On the stability of coagulation—fragmentation population balances},
  author={Vigil, R. D. and Ziff, R. M.},
  journal={Journal of Colloid and Interface Science},
  volume={133},
  number={1},
  pages={257--264},
  year={1989},
  publisher={Elsevier}
}

@book{seinfeld2016atmospheric,
  title={Atmospheric chemistry and physics: from air pollution to climate change},
  author={Seinfeld, J. H. and Pandis, S. N.},
  year={2016},
  publisher={John Wiley \& Sons}
}

@article{stewart,
  title={A global existence theorem for the general coagulation--fragmentation equation with unbounded kernels},
  author={Stewart, I. W. and Meister, E.},
  journal={Mathematical Methods in the Applied Sciences},
  volume={11},
  number={5},
  pages={627--648},
  year={1989},
  publisher={Wiley Online Library}
}

@article{carr1992asymptotic,
  title={Asymptotic behaviour of solutions to the coagulation--fragmentation equations. {I}. {T}he strong fragmentation case},
  author={Carr, J.},
  journal={Proceedings of the Royal Society of Edinburgh Section A: Mathematics},
  volume={121},
  number={3-4},
  pages={231--244},
  year={1992},
  publisher={Royal Society of Edinburgh Scotland Foundation}
}

@article{escobedo2005self,
  title={On self-similarity and stationary problem for fragmentation and coagulation models},
  author={Escobedo, M. and Mischler, S. and Ricard, M. R.},
  journal={Annales de l'Institut Henri Poincar{\'e} C},
  volume={22},
  number={1},
  pages={99--125},
  year={2005}
}

@article{hendriks1983coagulation,
  title={Coagulation equations with gelation},
  author={Hendriks, E. M. and Ernst, M. H. and Ziff, R. M.},
  journal={Journal of Statistical Physics},
  volume={31},
  pages={519--563},
  year={1983},
  publisher={Springer}
}

@article{mcgrady1987shattering,
  title={‘‘{S}hattering’’ transition in fragmentation},
  author={McGrady, E. D. and Ziff, R. M.},
  journal={Physical Review Letters},
  volume={58},
  number={9},
  pages={892},
  year={1987},
  publisher={APS}
}

@inproceedings{doumic2018estimating,
  title={Estimating the division rate and kernel in the fragmentation equation},
  author={Doumic, M. and Escobedo, M. and Tournus, M.},
  booktitle={Annales de l'Institut Henri Poincar{\'e} C, Analyse non lin{\'e}aire},
  volume={35},
  number={7},
  pages={1847--1884},
  year={2018},
  organization={Elsevier}
}

@article{alomari2013recovery,
  title={Recovery of the integral kernel in the kinetic fragmentation equation},
  author={Alomari, O. and Dubovski, P. B.},
  journal={Inverse Problems in Science and Engineering},
  volume={21},
  number={1},
  pages={171--181},
  year={2013},
  publisher={Taylor \& Francis}
}

@article{andrejevic2021model,
  title={A model for the fragmentation kinetics of crumpled thin sheets},
  author={Andrejevic, J. and Lee, L. M. and Rubinstein, S. M. and Rycroft, C. H.},
  journal={Nature Communications},
  volume={12},
  number={1},
  pages={1470},
  year={2021},
  publisher={Nature Publishing Group UK London}
}

@article{timar2010new,
  title={New universality class for the fragmentation of plastic materials},
  author={Tim{\'a}r, G. and Bl{\"o}mer, J. and Kun, F. and Herrmann, H. J.},
  journal={Physical Review Letters},
  volume={104},
  number={9},
  pages={095502},
  year={2010},
  publisher={APS}
}

@book{pazy2012semigroups,
  title={Semigroups of linear operators and applications to partial differential equations},
  author={Pazy, A.},
  volume={44},
  year={2012},
  publisher={Springer Science \& Business Media}
}

@book{hinze2008optimization,
  title={Optimization with {PDE} constraints},
  author={Hinze, M. and Pinnau, R. and Ulbrich, M. and Ulbrich, S.},
  volume={23},
  year={2009},
  publisher={Springer Science \& Business Media}
}

@article{kumar1996solution,
  title={On the solution of population balance equations by discretization—{I}. {A} fixed pivot technique},
  author={Kumar, S. and Ramkrishna, D.},
  journal={Chemical Engineering Science},
  volume={51},
  number={8},
  pages={1311--1332},
  year={1996},
  publisher={Elsevier}
}

@article{giri2012uniqueness,
  title   = {On the uniqueness for coagulation and multiple fragmentation equation},
  author  = {Giri, A. K.},
  journal = {Kinetic and Related Models},
  volume  = {6},
  number  = {3},
  pages   = {589--599},
  year    = {2013},
}

@article{doumic2021inverse,
  title={An inverse problem: recovering the fragmentation kernel from the short-time behaviour of the fragmentation equation},
  author={Doumic, M. and Escobedo, M. and Tournus, M.},
  journal={Annales Henri Lebesgue},
  volume={7},
  pages={621--671},
  year={2024},
  publisher={Centre Mersenne}
}

@article{coron2015optimization,
  title={Optimization of an amplification protocol for misfolded proteins by using relaxed control},
  author={Coron, J.-M. and Gabriel, P. and Shang, P.},
  journal={Journal of Mathematical Biology},
  volume={70},
  number={1-2},
  pages={289--327},
  year={2015},
  publisher={Springer}
}

@incollection{nagy2008distribution,
  title={A population balance model approach for crystallization product engineering via distribution shaping control},
  author={Nagy, Z. K.},
  booktitle={18th European Symposium on Computer Aided Process Engineering -- ESCAPE 18},
  series={Computer Aided Chemical Engineering},
  volume={25},
  pages={139--144},
  year={2008},
  publisher={Elsevier}
}

@article{chyba2015optimal,
  title={Optimal geometric control applied to the protein misfolding cyclic amplification process},
  author={Chyba, M. and Coron, J.-M. and Gabriel, P. and Jacquemard, A. and Patterson, G. and Picot, G. and Shang, P.},
  journal={Acta Applicandae Mathematicae},
  volume={135},
  number={1},
  pages={145--173},
  year={2015},
  publisher={Springer}
}

@book{engel2000one,
  title={One-parameter semigroups for linear evolution equations},
  author={Engel, K.-J. and Nagel, R.},
  year={2000},
  publisher={Springer}
}

@book{banasiak2006perturbations,
  title={Perturbations of positive semigroups with applications},
  author={Banasiak, J. and Arlotti, L.},
  year={2006},
  publisher={Springer}
}

@book{cazenave1998semilinear,
  title={An Introduction to Semilinear Evolution Equations},
  author={Cazenave, T. and Haraux, A.},
  year={1998},
  publisher={Clarendon Press}
}

@article{hofmann2017optimal,
  title={Optimal control of univariate and multivariate population balance systems involving external fines removal},
  author={Hofmann, S. and Bajcinca, N. and Raisch, J. and Sundmacher, K.},
  journal={Chemical Engineering Science},
  volume={168},
  pages={101--123},
  year={2017},
  publisher={Elsevier}
}

@article{de2021optimal,
  title={Optimal control of crystal size and shape in batch crystallization using a bivariate population balance modeling},
  author={de Moraes, M. G. F. and Grover, M. A. and de Souza Jr, M. B. and Lage, P. L. C. and Secchi, A. R.},
  journal={IFAC-PapersOnLine},
  volume={54},
  number={3},
  pages={653--660},
  year={2021},
  publisher={Elsevier}
}

\end{document}